\documentclass[12pt]{amsart}
\usepackage{a4wide}
\usepackage{enumerate}
\usepackage{graphicx,tikz}
\usepackage{color}
\usepackage{amsmath}
\usepackage{scrextend} 
\usepackage{mathabx} 
\allowdisplaybreaks

\let\pa\partial
\let\na\nabla
\let\eps\varepsilon
\newcommand{\N}{{\mathbb N}}
\newcommand{\R}{{\mathbb R}}
\newcommand{\diver}{\operatorname{div}}

\newcommand{\D}{{\mathcal D}}
\newcommand{\E}{{\mathcal E}}
\newcommand{\F}{{\mathcal F}}
\newcommand{\T}{{\mathcal T}}
\newcommand{\V}{{\mathcal V}}
\newcommand{\dist}{{\mathrm{d}}}
\newcommand{\m}{\mathrm{m}}
\newcommand{\dt}{{\triangle t}}

\newtheorem{theorem}{Theorem}
\newtheorem{lemma}[theorem]{Lemma}
\newtheorem{proposition}[theorem]{Proposition}
\newtheorem{remark}{Remark}

\begin{document}

\title[Convergence of a finite-volume scheme]{Convergence of a 
finite-volume scheme for a degenerate cross-diffusion model for ion transport}

\author[C. Canc\`es]{Cl\'ement Canc\`es}
\address{Inria, Univ. Lille, CNRS, UMR 8524 - Laboratoire Paul Painlev\'e, F-59000 Lille}
\email{clement.cances@inria.fr}

\author[C. Chainais-Hillairet]{Claire Chainais-Hillairet}
\address{Univ. Lille, CNRS, UMR 8524, Inria - Laboratoire Paul Painlev\'e, F-59000 Lille}
\email{Claire.Chainais@math.univ-lille1.fr}

\author[A. Gerstenmayer]{Anita Gerstenmayer}
\address{Institute for Analysis and Scientific Computing, Vienna University of
	Technology, Wiedner Hauptstra\ss e 8--10, 1040 Wien, Austria}
\email{anita.gerstenmayer@tuwien.ac.at}

\author[A. J\"ungel]{Ansgar J\"ungel}
\address{Institute for Analysis and Scientific Computing, Vienna University of
	Technology, Wiedner Hauptstra\ss e 8--10, 1040 Wien, Austria}
\email{juengel@tuwien.ac.at}

\date{\today}

\thanks{The authors have been supported by the Austrian-French Program {\em Amad\'ee}
of the Austrian Exchange Service (\"OAD). The work of the first and second authors is supported by
the LABEX CEMPI (ANR-11-LABX-0007-01).
The third and fourth authors acknowledge partial support from
the Austrian Science Fund (FWF), grants P27352, P30000, F65, and W1245}

\begin{abstract}
An implicit Euler finite-volume scheme for a degenerate cross-diffusion system 
describing the ion transport through biological membranes is analyzed. 
The strongly coupled equations for the ion concentrations include drift terms
involving the electric potential, which is coupled to the concentrations through
the Poisson equation. The cross-diffusion system possesses a
formal gradient-flow structure revealing nonstandard degeneracies,
which lead to considerable mathematical difficulties.
The finite-volume scheme is based on two-point flux
approximations with ``double'' upwind mobilities. It preserves the structure of
the continuous model like nonnegativity, upper bounds, and entropy dissipation. 
The degeneracy is overcome by proving a new discrete Aubin-Lions lemma 
of ``degenerate'' type.
Under suitable assumptions, the existence and uniqueness of bounded discrete solutions,
a discrete entropy inequality, and the convergence of the scheme
is proved. Numerical simulations of a calcium-selective ion channel 
in two space dimensions indicate that the numerical scheme is of first order.
\end{abstract}



\maketitle


\section{Introduction}

The ion transport through biological channels 
plays an important role in all living organisms.
On a macroscopic level, the transport can be described by nonlinear partial differential
equations for the ion concentrations (or, more precisely, volume fractions)
and the surrounding electric potential. A classical model for ion transport
are the Poisson-Nernst-Planck equations \cite{Ner88}, which satisfy Fick's law
for the fluxes. However, this approach does not include size exclusion effects
in narrow ion channels. Taking into account the finite size of the ions, one can
derive from an on-lattice model in the diffusion limit another set of 
differential equations with fluxes depending on the gradients of all species 
\cite{BDPS10,SLH09}.
These nonlinear cross-diffusion terms are common in multicomponent systems
\cite[Chapter~4]{Jue16}.
In this paper, we propose an implicit Euler finite-volume discretization
of the resulting cross-diffusion system. The scheme is designed in such a way that
the nonnegativity and upper bound of the concentrations as well as the 
entropy dissipation is preserved on the discrete level. 

More specifically, the evolution of the concentrations $u_i$ and fluxes $\F_i$ 
of the $i$th ion species is governed by the equations
\begin{equation}\label{1.eq}
  \pa_t u_i + \diver \F_i = 0, \quad \F_i = -D_i\big(u_0\na u_i - u_i\na u_0 
  + u_0u_i\beta z_i\na\Phi\big)\quad\mbox{in }\Omega,\ t>0,
\end{equation}
for $i=1,\ldots,n$, where $u_0=1-\sum_{i=1}^n u_i$ is the concentration (volume fraction)
of the electro-neutral solvent, $D_i>0$ is a diffusion coefficient, $\beta>0$
is the (scaled) inverse thermal voltage, and $z_i\in\R$ the charge of the $i$th species.
Observe that we assumed Einstein's relation which says that the quotient of the
diffusion and mobility coefficients is constant, and we call this constant $1/\beta$.
The electric potential is determined by the Poisson equation
\begin{equation}\label{1.poi}
  -\lambda^2\Delta\Phi = \sum_{i=1}^nz_iu_i + f \quad\mbox{in }\Omega,
\end{equation}
where $\lambda^2$ is the (scaled) permittivity constant and $f=f(x)$ is a permanent
background charge density. Equations \eqref{1.eq} and \eqref{1.poi}
are solved in a bounded domain $\Omega\subset\R^d$. 

In order to match experimental conditions, the boundary $\pa\Omega$ is supposed to
consist of two parts, the insulating part $\Gamma_N$, on which no-flux boundary
conditions are prescribed, and the union $\Gamma_D$ of boundary contacts
with external reservoirs, on which the concentrations are fixed. The electric
potential is prescribed at the electrodes on $\Gamma_D$. This leads to the
mixed Dirichlet-Neumann boundary conditions
\begin{align}
  \F_i\cdot\nu=0\quad\mbox{on }\Gamma_N, 
	&\quad u_i=\overline{u}_i\quad\mbox{on }\Gamma_D, \quad i=1,\ldots,n, \label{1.bc1} \\
  \na\Phi\cdot\nu=0\quad\mbox{on }\Gamma_N, 
	&\quad \Phi=\overline{\Phi}\quad\mbox{on }\Gamma_D, \label{1.bc2}
\end{align}
where the boundary data $(\overline{u}_i)_{1\leq i\leq n}$ and ${\overline{\Phi}}$ can be defined on the whole domain $\Omega$.
Finally, we prescribe the initial conditions
\begin{equation}\label{1.ic}
  u_i(\cdot,0) = u_i^{I}\quad\mbox{in }\Omega,\ i=1,\ldots,n.
\end{equation}

The main mathematical difficulties of equations \eqref{1.eq} are the strong coupling
and the fact that the diffusion matrix $(A_{ij}(u))$, defined by
$A_{ij}(u)=D_iu_i$ for $i\neq j$ and $A_{ii}(u)=D_i(u_0+u_i)$ is not symmetric
and not positive definite. It was shown in \cite{BDPS10,Jue15} that
system \eqref{1.eq} possesses a formal gradient-flow structure. This means that
there exists a (relative) entropy functional $H[u]=\int_\Omega h(u)dx$ with the
entropy density
$$
  h(u) = \sum_{i=0}^n\int_{\overline{u}_i}^{u_i}\log\frac{s}{\overline{u}_i}ds
	+ \frac{\beta\lambda^2}{2}|\na(\Phi-\overline{\Phi})|^2,
$$
where $u=(u_1,\ldots,u_n)$ and $u_0=1-\sum_{i=1}^n u_i$, 
such that \eqref{1.eq} can be formally written as
$$
  \pa_t u_i = \diver\bigg(\sum_{j=1}^n B_{ij}\na w_j\bigg),
$$
where $B_{ii}=D_iu_0u_i$, $B_{ij}=0$ for $i\neq j$ provide a diagonal positive definite
matrix, and $w_j$ are the entropy variables, defined by
\begin{align*}
  & \frac{\pa h}{\pa u_i} = w_i - \overline{w}_i, \quad\mbox{where} \\
  & w_i = \log\frac{u_i}{u_0} + \beta z_i\Phi, \quad 
	\overline{w}_i = \log\frac{\overline{u}_i}{\overline{u}_0} + \beta z_i\overline{\Phi},
	\quad i=1,\ldots,n.
\end{align*}
We refer to \cite[Lemma 7]{GeJu17} for the computation of $\pa h/\pa u_i$.

The entropy structure of \eqref{1.eq} is useful for two reasons. First,
it leads to $L^\infty$ bounds for the concentrations. Indeed, 
the transformation $(u,\Phi)\mapsto w$ to entropy variables can be inverted,
giving $u=u(w,\Phi)$ with
$$
  u_i(w,\Phi) = \frac{\exp(w_i-\beta z_i\Phi)}{1+\sum_{j=1}^n
	\exp(w_j-\beta z_j\Phi)}, \quad i=1,\ldots,n.
$$
Then $u_i$ is positive and bounded from above, i.e.\ 
\begin{equation}\label{1.D}
  u\in\D = \bigg\{u\in(0,1)^n:\sum_{i=1}^n u_i < 1\bigg\}.
\end{equation}
This yields $L^\infty$ bounds without the use of a maximum principle. 
Second, the entropy
structure leads to gradient estimates via the entropy inequality
$$
  \frac{dH}{dt} + \frac12\int_\Omega \sum_{i=1}^n D_iu_0u_i|\na w_i|^2 dx \le C,
$$
where the constant $C>0$ depends on the Dirichlet boundary data.
Because of 
\begin{equation}\label{1.naw}
  \sum_{i=1}^n u_0u_i\bigg|\na \log\frac{u_i}{u_0}\bigg|^2 
	= 4\sum_{i=1}^n u_0|\na u_i^{1/2}|^2 + 4|\na u_0^{1/2}|^2 + |\na u_0|^2,
\end{equation}
we achieve gradient estimates for $u_0^{1/2}u_i$ and $u_0^{1/2}$.
Since $u_0$ may vanish locally, this does not give gradient bounds for $u_i$,
which expresses the degenerate nature of the cross-diffusion system.
As a consequence, the flux has to be formulated in the terms of gradients
of $u_0^{1/2}u_i$ and $u_0^{1/2}$ only, namely
\begin{equation}\label{1.Fsqrt}
  \F_i = -D_i\big(u_0^{1/2}\na(u_0^{1/2}u_i) - 3u_0^{1/2}u_i\na u_0^{1/2}
	+ u_0u_i\beta z_i\na\Phi\big).
\end{equation}
The challenge is to derive a discrete version of this formulation.
It turns out that \eqref{2.Fsqrt} below is the right formulation in our context
(assuming vanishing drift parts). 

Our aim is to design a numerical approximation of \eqref{1.eq} which preserves the
structural properties of the continuous equations. This suggests to use
the entropy variables as the unknowns, as it was done in our previous work
\cite{GeJu17} with simulations in one space dimension.
Unfortunately, we have not been able to perform a numerical convergence analysis
with these variables. The reason is that we need discrete chain rules in order
to formulate \eqref{1.naw} on the discrete level and these discrete chain rules
seem to be not easily available. Therefore, we use the original variables $u_i$ for the
numerical discretization. Interestingly, we are still able to prove that the
scheme preserves the nonnegativity, upper bound, and entropy inequality.
However, the upper bound comes at a price: We need to assume that all
diffusion coefficients $D_i$ are the same. Under this assumption, 
$u_0=1-\sum_{i=1}^n u_i$ solves a drift-diffusion equation for which the
(discrete) maximum principle can be applied. It is not surprising that the
$L^\infty$ bound can be shown only under an additional condition, since 
cross-diffusion systems usually do not allow for a maximum principle.

The key observation for the numerical discretization is that the fluxes can
be written on each cell in  a ``double'' drift-diffusion form, i.e., both
$\F_i = -D_i(u_0\na u_i - u_i V_i)$ and $V_i = \na u_0 - \beta z_i u_0\na\Phi$
have the structure $\na v + vF$, where $\na v$ is the diffusion term and
$vF$ is the drift term. We discretize $\F$ and $V$ by using a two-point flux
aproximation with ``double'' upwind mobilities.

Our analytical results are stated and proved for no-flux boundary conditions on
$\pa\Omega$. Mixed Dirichlet-Neumann boundary conditions could be prescribed as
well, but the proofs would become even more technical.
The main results are as follows. 

\begin{itemize}
\item We prove the existence of solutions to the fully discrete numerical scheme
(Theorem \ref{thm.ex}). If the drift part vanishes, the solution is unique. 
The existence proof uses a topological degree argument in finite space dimensions, 
while the uniqueness proof is based on the entropy method of Gajewski \cite{Gaj94}, 
recently extended to cross-diffusion systems \cite{GeJu17,ZaJu17}.
\item Thanks to the ``double'' upwind structure, 
the scheme preserves the nonnegativity and upper bound for the concentrations
(at least if $D_i=D$ for all $i$). 
Moreover, convexity arguments show that the discrete entropy is dissipated with
a discrete entropy production analogous to \eqref{1.naw}
(Theorem \ref{thm.ent}). The proof of the discrete entropy 
inequality only works if the drift term vanishes, since
we need to control a discrete version of the sum $\sum_{i=1}^n u_i$ from below;
see the discussion after Theorem \ref{thm.ent}.
\item The discrete solutions converge to the continuous solutions to \eqref{1.eq}
as the mesh size tends to zero (Theorem \ref{thm.conv}). The proof is based on
a priori estimates obtained from the discrete entropy inequality. 
The compactness is derived from a new discrete Aubin-Lions lemma, which takes 
into account the nonstandard degeneracy of the equations; see Lemma \ref{lem.aubin2}
in the appendix.
\item Numerical experiments for a calcium-selective ion channel in two space dimensions
show the dynamical behavior of the solutions and their large-time asymptotics to the
equilibrium. The tests indicate that the order of convergence in the $L^1$ norm
is one.
\end{itemize}

In the literature, there exist some results on finite-volume schemes for
cross-diffusion systems. 
An upwind two-point flux approximation similar to
our discretization was recently used in \cite{Oul17} for a seawater intrusion 
cross-diffusion model. 
A two-point flux approximation with a nonlinear positivity-preserving approximation
of the cross-diffusion coefficients, modeling the segregation of a two-species
population, was suggested in \cite{ABB11}, assuming positive definiteness of
the diffusion matrix. The Laplacian structure of the population
model (still for positive definite matrices) was exploited in \cite{Mur17}
to design a convergent {\em linear} finite-volume scheme, which avoids fully 
implicit approximations.
A semi-implicit finite-volume discretization for a biofilm model with a nonlocal
time integrator was proposed in \cite{RaEb14}.
Finite-volume schemes for cross-diffusion systems with nonlocal (in space) terms were
also analyzed; see, for instance, \cite{ABS15} for a food chain model and 
\cite{ABLS15} for an epidemic model. Moreover, a finite-volume
scheme for a Keller-Segel system with additional cross diffusion and discrete
entropy dissipation property was investigated in \cite{BeJu14}.
All these models, however, do not include volume filling and do not possess the
degenerate structure explained before.

The paper is organized as follows. The numerical scheme and the main results
are presented in Section \ref{sec.scheme}. In Section \ref{sec.ex}, the 
existence and uniqueness of bounded discrete solutions are shown.
We prove the discrete entropy inequality and further a priori estimates
in Section \ref{sec.est}, while Section \ref{sec.conv} is concerned with
the convergence of the numerical scheme. Numerical experiments are given in Section \ref{sec.num} in order to illustrate the order of convergence and the long time behavior of the scheme. 
 For the compactness arguments, we need
two discrete Aubin-Lions lemmas which are proved in the appendix. 


\section{Numerical scheme and main results}\label{sec.scheme}

\subsection{Notations and definitions}

We summarize our general hypotheses on the data:

\begin{labeling}{(A44)}
\item[(H1)] Domain: $\Omega\subset\R^d$ ($d=2$ or $d=3$) is an open, bounded, 
polygonal domain with $\pa\Omega=\Gamma_D\cup\Gamma_N\in C^{0,1}$, 
$\Gamma_D\cap\Gamma_N=\emptyset$.

\item[(H2)] Parameters: $T>0$, $D_i>0$, $\beta>0$, and $z_i\in\R$,
$i=1,\ldots,n$.

\item[(H3)] Background charge: $f\in L^\infty(\Omega)$.

\item[(H4)] Initial and boundary data: $u_i^{I}\in L^\infty(\Omega)$,
$\overline{u}_i\in H^1(\Omega)$ satisfy $u_i^{I}\ge 0$, $\overline{u}_i \ge 0$ and 
$1-\sum_{i=1}^n u_i^{I}\ge 0$, $1-\sum_{i=1}^n \overline{u}_i\ge 0$
in $\Omega$ for $i=1,\ldots,n$, 
and $\overline{\Phi}\in H^1(\Omega)\cap L^\infty(\Omega)$.
\end{labeling}

For our main results, we need additional technical assumptions:

\begin{labeling}{(A44)}
	\item[(A1)] $\pa\Omega=\Gamma_N$, i.e., we impose no-flux
	boundary conditions on the whole boundary.
	\item[(A2)] The diffusion constants are equal, $D_i=D>0$ for $i=1,\ldots,n$.
	\item[(A3)] The drift terms are set to zero, $\Phi\equiv 0$.
\end{labeling}

\begin{remark}[Discussion of the assumptions]\rm\label{rem.disc}
Assumption (A1) is supposed for simplicity only. Mixed Dirichlet-Neumann
boundary conditions can be included in the analysis (see, e.g., \cite{GeJu17}), 
but the proofs become even more technical. 
Mixed boundary conditions are chosen in the numerical experiments; therefore,
the numerical scheme is defined for that case.
Assumption (A2) is needed for the derivation of an upper bound for the 
solvent concentration.
Indeed, when $D_i=D$ for all $i$, summing \eqref{1.eq} over $i=1,\ldots,n$ gives
$$
  \pa_t u_0 = D\diver(\na u_0 - u_0 w\na\Phi), \quad\mbox{where } 
	w = \beta\sum_{i=1}^n z_iu_i.
$$
On the discrete level, we replace $u_0w\na\Phi$ by an upwind approximation. This
allows us to apply the discrete maximum principle showing that $u_0\ge 0$
and hence $u=(u_1,\ldots,u_n)\in\overline\D$ with $\D$ defined in \eqref{1.D}.
Finally, Assumption (A3) is needed to derive a discrete version of the entropy inequality.
Without the drift terms, the upwinding value does not depend on the index of the
species, which simplifies some expressions; see Remark \ref{rem.simpl}. 
\qed
\end{remark}

For the definition of the numerical scheme for \eqref{1.eq}-\eqref{1.poi}, we need
to introduce a suitable discretization of the domain $\Omega$ and the interval
$(0,T)$. For simplicity, we consider a uniform time discretization with time
step $\dt>0$, and we set $t^k=k\dt$ for $k=1,\ldots,N$, where $T>0$,
$N\in\N$ are given and $\dt=T/N$. The domain $\Omega$ is discretized by a
regular and admissible triangulation in the sense of \cite[Definition 9.1]{EGH00}.
The triangulation consists of a family $\T$ of open polygonal convex subsets of 
$\Omega$ (so-called cells), a family $\E$ of edges (or faces in three dimensions), 
and a family of points $(x_K)_{K\in\T}$ associated to the cells. 
The admissibility assumption implies that the straight line between two 
centers of neighboring cells $\overline{x_Kx_L}$ 
is orthogonal to the edge $\sigma=K|L$ between two
cells $K$ and $L$. The condition is satisfied by, for instance, triangular meshes
whose triangles have angles smaller than $\pi/2$ \cite[Examples 9.1]{EGH00}
or Voronoi meshes \cite[Example 9.2]{EGH00}.

We assume that the family of edges
$\E$ can be split into internal and external edges $\E=\E_{\rm int}\cup\E_{\rm ext}$
with $\E_{\rm int} = \{\sigma\in\E:\sigma\subset\Omega\}$ and $\E_{\rm ext}=\{\sigma
\in\E:\sigma\subset\pa\Omega\}$. Each exterior edge is assumed to be an element of
either the Dirichlet or Neumann boundary, i.e.\ $\E_{\rm ext}=\E_{\rm ext}^D\cup
\E_{\rm ext}^N$. 
For given $K\in\T$, we define the set $\E_K$ of the edges of $K$, which is the union
of internal edges and edges on the Dirichlet or Neumann boundary, and we set 
$\E_{K,\rm int}=\E_K\cap \E_{\rm int}$.

The size of the mesh
is defined by $h(\T) = \sup\{\operatorname{diam}(K):K\in\T\}$.  
For $\sigma\in\E_{\rm int}$ with $\sigma=K|L$, 
we denote by $\dist_\sigma=\dist(x_K,x_L)$ 
the Euclidean distance between $x_K$ and $x_L$, while
for $\sigma\in\E_{\rm ext}$, we set $\dist_\sigma=d(x_K,\sigma)$.
For a given edge $\sigma\in\E$,
the transmissibility coefficient is defined by
\begin{equation}\label{2.trans}
  \tau_\sigma = \frac{\m(\sigma)}{\dist_\sigma},
\end{equation}
where $\m(\sigma)$ denotes the Lebesgue measure of $\sigma$. 

We impose a regularity assumption on the mesh: 
There exists $\zeta>0$ such that for all $K\in\T$ and $\sigma\in\E_K$, 
it holds that 
\begin{equation}\label{2.dd}
  \dist(x_K,\sigma)\ge \zeta \dist_\sigma.
\end{equation}
This hypothesis is needed to apply discrete functional inequalities 
(see \cite{BCF15,EGH00}) and a discrete compactness theorem 
(see \cite{GaLa12}). 

It remains to introduce suitable function spaces for the numerical discretization.
The space $\mathcal{H}_\T$ of piecewise constant functions is defined by
$$
  \mathcal{H}_\T = \bigg\{v:\overline\Omega\to\R:\exists (v_K)_{K\in\T}\subset\R,\
	v(x)=\sum_{K\in\T}v_K\mathbf{1}_K(x)\bigg\}.
$$
The (squared) discrete $H^1$ norm on this space is given by
\begin{equation}\label{2.norm1}
  \|v\|_{1,\T}^2 = \sum_{\sigma=K|L\in\E_{\rm int}}\tau_\sigma(v_K-v_L)^2
	+ \sum_{K\in\T}\m(K)v_K^2.
\end{equation}
The discrete $H^{-1}$ norm is the dual norm with respect to the $L^2$
scalar product,
\begin{equation}\label{2.norm-1}
  \|v\|_{-1,\T} = \sup\bigg\{\int_\Omega vw dx: w\in\mathcal{H}_\T,\
	\|w\|_{1,\T}=1\bigg\}.
\end{equation}
Then 
$$
  \bigg|\int_\Omega vw dx\bigg| \le \|v\|_{-1,\T}\|w\|_{1,\T}
	\quad\mbox{for }v,w\in\mathcal{H}_\T.
$$
Finally, we introduce the space $\mathcal{H}_{\T,\dt}$ of piecewise constant
in time functions with values in $\mathcal{H}_\T$,
$$
  \mathcal{H}_{\T,\dt} = \bigg\{ v:\overline\Omega\times[0,T]\to\R:
	\exists(v^k)_{k=1,\ldots,N}\subset\mathcal{H}_\T,\
	v(x,t) = \sum_{k=1}^Nv^k(x)\mathbf{1}_{(t^{k-1},t^k)}(t)\bigg\},
$$
equipped with the discrete $L^2(0,T;H^1(\Omega))$ norm
$$
  \|v\|_{1,\T,\dt} = \bigg(\sum_{k=1}^N\dt\|v^k\|_{1,\T}^2\bigg)^{1/2}.
$$

For the numerical scheme, we introduce some further definitions.
Let $u_i\in\mathcal{H}_\T$ with values $\overline{u}_{i,\sigma}$ on the Dirichlet
boundary ($i=1,\ldots,n$). Then we introduce
\begin{align}
	&\text{D}_{K,\sigma}(u_i) = u_{i,K,\sigma}-u_{i,K}, \label{2.defDu} \\
  &\mbox{where}\quad u_{i,K,\sigma} = \begin{cases}
	u_{i,L} \quad&\text{ for }\sigma\in\E_{\rm int},\ \sigma=K|L,\\
  \overline{u}_{i,\sigma} \quad&\text{ for }\sigma\in\E^D_{{\rm ext},K},\\
  u_{i,K} &\text{ for }\sigma\in\E^N_{{\rm ext},K}, 
  \end{cases} \quad
	\overline{u}_{i,\sigma} = \frac{1}{\m(\sigma)}\int_\sigma\overline{u}_i ds.
	\nonumber
\end{align}
The numerical fluxes $\F_{K,\sigma}$
should be consistent approximations to the exact fluxes through
the edges $\int_\sigma \F\cdot\nu ds$. We impose the conservation of the numerical
fluxes $\F_{K,\sigma}+\F_{L,\sigma}=0$ for edges $\sigma=K|L$, requiring that they
vanish on the Neumann boundary edges, $\F_{K,\sigma}=0$ for 
$\sigma\in\E_{{\rm ext},K}^N$. Then the discrete integration-by-parts formula becomes
for $u\in\mathcal{H}_\T$
\begin{equation*}
  \sum_{K\in\T}\sum_{\sigma\in\E_K}\F_{K,\sigma}u_{K}
	= -\sum_{\sigma\in\E}\F_{K,\sigma}\textrm{D}_{K,\sigma}(u)
	+ \sum_{\sigma\in\E_{\rm ext}^D}\F_{K,\sigma}u_{K,\sigma}.
\end{equation*}
When $\pa\Omega=\Gamma_N$, this formula simplifies to
\begin{equation}\label{2.ibp}
  \sum_{K\in\T}\sum_{\sigma\in\E_K}\F_{K,\sigma}u_{K}
	= \sum_{\sigma=K|L\in\E_{\rm int}}\F_{K,\sigma}(u_{K}-u_{L}).
\end{equation}


\subsection{Numerical scheme}

We need to approximate the initial, boundary, and given functions on the elements
$K\in\T$ and edges $\sigma\in\E$:
\begin{align*}
  u^{I}_{i,K} &= \frac{1}{\m(K)}\int_{K}u^{I}_i(x)dx, &
	f_K &= \frac{1}{\m(K)}\int_{K}f(x)dx, \\
  \overline{u}_{i,\sigma} &= \frac{1}{\m(\sigma)}\int_{\sigma}\overline{u}_ids, & 
	\overline{\Phi}_{\sigma} &= \frac{1}{\m(\sigma)}\int_{\sigma}\overline{\Phi}ds,
\end{align*}
and we set $u_{0,K}^I= 1-\sum_{i=1}^n u_{i,K}^I$ and 
$\overline{u}_{0,\sigma}=1-\sum_{i=1}^n\overline{u}_{i,\sigma}$.

The numerical scheme is as follows. Let $K\in\T$, $k\in\{1,\ldots,N\}$, 
$i=1,\ldots,n$, and
$u_{i,K}^{k-1}\ge 0$ be given. Then the values $u_{i,K}^k$ are determined by
the implicit Euler scheme
\begin{equation}\label{2.equ}
  \m(K)\frac{u_{i,K}^k-u_{i,K}^{k-1}}{\dt} + \sum_{\sigma\in\E_K}\F_{i,K,\sigma}^k = 0,
\end{equation}
where the fluxes $\F_{i,K,\sigma}^k$ are given by the upwind scheme
\begin{equation}\label{2.eqF}
  \F_{i,K,\sigma}^k = -\tau_\sigma D_i\Big(u_{0,\sigma}^k\text{D}_{K,\sigma}(u_i^k)
	- u_{i,\sigma}^k\big(\text{D}_{K,\sigma}(u_0^k) - \widehat{u}_{0,\sigma,i}^k
	\beta z_i\text{D}_{K,\sigma}(\Phi^k)\big)\Big),
\end{equation}
where $\tau_\sigma$ is defined in \eqref{2.trans}, 
\begin{align}
  & u_{0,K}^k=1-\sum_{i=1}^n u_{i,K}^k, \quad 
	u_{0,\sigma}^k = \max\{u_{0,K}^k,u_{0,L}^k\}, \label{2.u0K} \\
	& u^k_{i,\sigma} = \begin{cases}
  u^k_{i,K} \quad & \text{if } \V^k_{i,K,\sigma}\ge 0, \\
  u^k_{i,K,\sigma} & \text{if }\V^k_{i,K,\sigma}< 0, 
  \end{cases}, \quad
  \widehat{u}_{0,\sigma,i}^k = \begin{cases}
  u^k_{0,K} \quad & \text{if }z_i\text{D}_{K,\sigma}(\Phi^k)\ge 0, \\
  u^k_{0,K,\sigma} & \text{if }z_i\text{D}_{K,\sigma}(\Phi^k)< 0, 
  \end{cases},  \label{2.hat}
\end{align}
and $\V_{i,K,\sigma}^k$ is the ``drift part'' of the flux,
\begin{equation}\label{2.V}
  \V_{i,K,\sigma}^k = \text{D}_{K,\sigma}(u_0^k) - \widehat{u}_{0,\sigma,i}^k
	\beta z_i\text{D}_{K,\sigma}(\Phi^k)
\end{equation}
for $i=1,\ldots,n$. Observe that we employed a double upwinding: one related to
the electric potential, defining $\widehat{u}_{0,\sigma,i}^k$, and
another one related to the drift part of the flux, $\V_{i,K,\sigma}^k$.
The potential is computed via
\begin{equation}\label{2.poi}
  -\lambda^2\sum_{\sigma\in\E_K}\tau_\sigma\text{D}_{K,\sigma}(\Phi^k)
	= \m(K)\bigg(\sum_{i=1}^n z_iu_{i,K}^k + f_K\bigg).
\end{equation}
We recall that the numerical boundary conditions are given by
$\overline{u}_{i,\sigma}$ and $\overline{\Phi}_\sigma$ for $\sigma\in\E_{\rm ext}^D$.

We denote by $u_{i,\T,\dt}$, $\Phi_{\T,\dt}$ the functions in 
$\mathcal{H}_{\T,\dt}$ associated to the values $u_{i,K}^k$ and $\Phi_K^k$,
respectively. Moreover, when dealing with a sequence of meshes $(\T_m)_{m}$ and a sequence of time steps $(\dt_m)_m$, we set $u_{i,m}=u_{i,\T_m,\dt_m}$, $\Phi_m=\Phi_{\T_m,\dt_m}$.

\begin{remark}[Simplified numerical scheme]\label{rem.simpl}\rm
When Assumptions (A1)-(A3) hold, the numerical scheme simplifies to
\begin{align}
  & \m(K)\frac{u_{i,K}^k-u_{i,K}^{k-1}}{\dt} + \sum_{\sigma\in\E_{K,\rm int}}\F_{i,K,\sigma}^k = 0,
	\label{2.simpl1} \\
  & \F_{i,K,\sigma}^k = -\tau_\sigma D\Big(u_{0,\sigma}^k\big(u_{i,L}^k-u_{i,K}^k\big)
	- u_{i,\sigma}^k\big(u_{0,L}^k-u_{0,K}^k\big)\Big), \label{2.simpl2}
\end{align}
where $u_{0,K}^k$ and $u_{0,\sigma}^k$ are defined in \eqref{2.u0K}, and the definition
of $u_{i,\sigma}^k$ simplifies to
$$
  u_{i,\sigma}^k = \begin{cases}
	u_{i,K}^k &\quad\text{if }u_{0,K}^k-u_{0,L}^k\le 0, \\
	u_{i,L}^k &\quad\text{if }u_{0,K}^k-u_{0,L}^k > 0.
	\end{cases}
$$
In the definition of $u_{i,\sigma}^k$, the upwinding value does not depend
on $i$ anymore such that 
\begin{equation}\label{2.sum}
  \sum_{i=0}^n u_{i,\sigma}^k = 1 + \max\{u_{0,K}^k,u_{0,L}^k\}
	- \min\{u_{0,K}^k,u_{0,L}^k\} = 1+|u_{0,K}^k-u_{0,L}^k|.
\end{equation}
This property is needed to control the sum $\sum_{i=1}^n u_{i,\sigma}^k$ from below
in the proof of the discrete entropy inequality; see \eqref{4.sum}.
Finally, we are able to reformulate the discrete fluxes such that we obtain
a discrete version of \eqref{1.Fsqrt} (without the drift part):
\begin{equation}\label{2.Fsqrt}
  \F_{i,K,\sigma} = \tau_\sigma D\bigg\{u_{0,\sigma}^{1/2}
	\big(u_{0,K}^{1/2}u_{i,K} - u_{0,L}^{1/2}u_{i,L}\big) 
	- u_{i,\sigma}\big(u_{0,K}^{1/2}-u_{0,L}^{1/2}\big)\bigg(
	u_{0,\sigma}^{1/2} + 2\frac{u_{0,K}^{1/2}+u_{0,L}^{1/2}}{2}\bigg)\bigg\}.
\end{equation}
This formulation is needed in the convergence analysis.
\qed
\end{remark}


\subsection{Main results}

Since our scheme is implicit and nonlinear, the existence of an approximate
solution is nontrivial. Therefore, our first result concerns the well-posedness
of the numerical scheme.

\begin{theorem}[Existence and uniqueness of solutions]\label{thm.ex}
Let (H1)-(H4) and (A2) hold. Then there exists a solution $(u,\Phi)$ to scheme
\eqref{2.equ}-\eqref{2.poi} satisfying $u^k\in\overline{\D}$ and, if the initial
data lie in $\D$, $u^k\in\D$. If additionally Assumptions (A1) and (A3) hold,
the solution is unique.
\end{theorem}

Assumption (A2) is needed to show that $u_0^k=1-\sum_{i=1}^n u_i^k$ is nonnegative.
Indeed, summing \eqref{2.equ} and \eqref{2.eqF} over $i=1,\ldots.n$, we obtain
$$
  \m(K)\frac{u_{0,K}^k-u_{0,K}^{k-1}}{\dt} 
	= -\sum_{\sigma\in\E_K}\tau_\sigma\bigg(u_{0,\sigma}^k\textrm{D}_{K,\sigma}
	\bigg(\sum_{i=1}^n D_iu_i^k\bigg) - \sum_{i=1}^n D_iu_{i,\sigma}^k
	\mathcal{V}_{i,K,\sigma}^k\bigg).
$$
Under Assumption (A2), it follows that $\sum_{i=1}^n D_iu_{i,K}^k=D(1-u_{0,K}^k)$,
and we can apply the discrete minimum principle, which then implies an $L^\infty$
bound for $u_{i}^k$. This bound allows us to apply a topological degree
argument; see \cite{Dei85,EGGH98}. For the uniqueness proof, we
additionally need Assumption (A3), since we use the entropy method of Gajewski
\cite{Gaj94}, and it seems that this method cannot be applied to
cross-diffusion systems including drift terms \cite{ZaJu17}. 
The idea is to prove first the uniqueness of $u_0^k$, which
solves a discrete nonlinear equation, and then to show the uniqueness of
$u_i^k$ for $i=1,\ldots,n$ by introducing a semimetric $d(u^k,v^k)$ for two
solutions $u^k=(u_1^k,\ldots,u_n^k)$ and $v^k=(v_1^k,\ldots,v_n^k)$ and showing
that it is monotone in $k$, such that a discrete Gronwall argument implies that
$u^k=v^k$.

The second result shows that the scheme preserves a discrete version of the
entropy inequality.

\begin{theorem}[Discrete entropy inequality]\label{thm.ent}\sloppy
Let Assumptions (H1)-(H4) and (A1)-(A3) hold. Then the solution to 
scheme \eqref{2.simpl1}-\eqref{2.simpl2} constructed in Theorem \ref{thm.ex}
satisfies the discrete entropy inequality
\begin{equation}\label{1.epi}
  \frac{H^k-H^{k-1}}{\dt} + I^k \le 0,
\end{equation}
with the discrete entropy
\begin{equation}\label{1.H}
  H^k = \sum_{K\in\T}\m(K)\sum_{i=0}^n\big(u_{i,K}^k(\log u_{i,K}^k-1)+1\big)
\end{equation}
and the discrete entropy production
\begin{align*}
  I^k &= D\sum_{\sigma=K|L\in\E_{\rm int}}\tau_\sigma\bigg(4\sum_{i=1}^n u_{0,\sigma}^k
	\big((u_{i,K}^k)^{1/2}-(u_{i,L}^k)^{1/2}\big)^2 \\
	&\phantom{xx}{}+ 4\big((u_{0,K}^k)^{1/2}-(u_{0,L}^k)^{1/2}\big)^2 
	+ \big(u_{0,K}^k-u_{0,L}^k\big)^2\bigg).
\end{align*}
\end{theorem}

Assumption (A3) is required to estimate the expression $\sum_{i=1}^n u_{i,\sigma}^k$.
In the continuous case, this sum equals $1-u_0$. On the discrete level, this
identity cannot be expected since the value of $u_{i,\sigma}^k$ depends on the
upwinding value; see \eqref{2.hat}. If the drift part vanishes, the upwinding
value does not depend on $i$, as mentioned in Remark \ref{rem.simpl}, and we can
derive the estimate $\sum_{i=1}^n u_{i,\sigma}^k\ge 1-u_{0,\sigma}^k$; see
Section \ref{sec.dei}. Note that the entropy production $I^k$ is the discrete
counterpart of \eqref{1.naw}.

The main result of this paper is the convergence of the approximate solutions
to a solution to the continuous cross-diffusion system.

\begin{theorem}[Convergence of the approximate solution]\label{thm.conv}\sloppy
Let (H1)-(H4) and (A1)-(A3) hold and 
let $(\T_m)$ and $(\dt_m)$ be sequences of admissible meshes and time steps,
respectively, such that $h(\T_m)\to 0$ and $\dt_m\to 0$ as $m\to\infty$.
Let $(u_{0,m},\ldots,u_{n,m})$ be the solution to \eqref{2.simpl1}-\eqref{2.simpl2}
constructed in Theorem \ref{thm.ex}.
Then there exist functions $u_0$, $u=(u_1,\ldots,u_n)$
satisfying $u(x,t)\in\overline\D$,
\begin{align*}
  & u_0^{1/2},\ u_0^{1/2}u_i \in L^2(0,T;H^1(\Omega)), \quad i=1,\ldots,n, \\
  & u_{0,m}^{1/2}\to u_0^{1/2},\ u_{0,m}^{1/2}u_{i,m}\to u_0^{1/2}u_i
	\quad\mbox{strongly in }L^2(\Omega\times(0,T)),
\end{align*}
where $u$ is a weak solution to
\eqref{1.eq}, \eqref{1.bc1}-\eqref{1.ic} (with $\Gamma_N=\pa\Omega$), i.e., for all 
$\phi\in C_0^\infty(\overline\Omega\times[0,T))$ and $i=1,\ldots,n$,
\begin{equation}\label{1.weak}
  \int_0^T\int_\Omega u_i\pa_t\phi dxdt + \int_\Omega u_i^I\phi(\cdot,0)dx
	= D\int_0^T\int_\Omega u_0^{1/2}\big(\na(u_0^{1/2}u_i) - 3u_i\na u_0^{1/2}\big)
	\cdot\na\phi dxdt.
\end{equation}
\end{theorem}

The compactness of the concentrations follows from the discrete gradient estimates
derived from the entropy inequality \eqref{1.epi}, for which we need Assumption (A3). 
By the discrete Aubin-Lions lemma \cite{EGH08}, we conclude the strong convergence
of the sequence $(u_{0,m}^{1/2})$.
The difficult part is to show the strong convergence of $(u_{0,m}^{1/2}u_{i,m})$,
since there is no control on the discrete gradient of $u_{i,m}$. The idea is
to apply a discrete Aubin-Lions lemma of ``degenerate'' type, proved in
Lemma \ref{lem.aubin2} in the appendix. 


\section{Existence and uniqueness of approximate solutions}\label{sec.ex}

\subsection{$L^\infty$ bounds and existence of solutions}

In order to prove the existence of solutions to \eqref{2.equ}-\eqref{2.poi},
we first consider a truncated problem. This means that we truncate the expressions
in \eqref{2.hat}; more precisely, we consider scheme \eqref{2.equ}, \eqref{2.eqF},
and \eqref{2.poi} with
\begin{align}
  & u^k_{0,K} = 1-\sum_{i=1}^n (u^k_{i,K})^+, \quad 
	u^k_{0,\sigma} = \max\{0,u^k_{0,K},u^k_{0,K,\sigma}\}, \nonumber \\ 
  & \widehat{u}_{0,\sigma,i}^k=\begin{cases}
    (u^k_{0,K})^+ \quad & \text{if }z_i\text{D}_{K,\sigma}(\Phi^k)\ge 0, \\
    (u^k_{0,K,\sigma})^+ & \text{if }z_i\text{D}_{K,\sigma}(\Phi^k)< 0,
  \end{cases} \label{3.trunc} \\
  & u^k_{i,\sigma} = \begin{cases}
    (u^k_{i,K})^+ \quad & \text{if } \V^k_{i,K,\sigma}\ge 0, \\
    (u^k_{i,K,\sigma})^+ & \text{if }\V^k_{i,K,\sigma}< 0, 
  \end{cases} \nonumber
\end{align}
where $z^+=\max\{0,z\}$ for $z\in\R$ and $i=1,\ldots,n$. We show that this
truncation is, in fact, not needed if the initial data are nonnegative.
In the following let (H1)-(H4) hold. 

\begin{lemma}[Nonnegativity of $u_i^k$]\label{lem.ui}
Let $(u,\Phi)$ be a solution to 
\eqref{2.equ}, \eqref{2.eqF}, \eqref{2.poi}, and 
\eqref{3.trunc}. Then $u_{i,K}^k\ge 0$
for all $K\in\T$, $k\in\{1,\ldots,N\}$, and $i=1,\ldots,n$. If $u_i^I>0$ and
$\overline{u}_i>0$ then also $u_{i,K}^k>0$ for all $K\in\T$, $k\in\{1,\ldots,N\}$.
\end{lemma}

\begin{proof}
We proceed by induction. For $k=0$, the nonnegativity holds because of our
assumptions on the initial data. Assume that $u_{i,L}^{k-1}\ge 0$ for all
$L\in\T$. Then let $u_{i,K}^k = \min\{u_{i,L}^k:L\in\T\}$ for some
$K\in\T$ and assume that $u_{i,K}^k<0$. The scheme writes as 
\begin{equation}\label{3.aux}
  \m(K)\frac{u_{i,K}^k-u_{i,K}^{k-1}}{\dt}
	= \sum_{\sigma\in\E_K}\tau_\sigma D_i\big(u_{0,\sigma}^k\text{D}_{K,\sigma}(u_i^k)
	- u_{i,\sigma}^k\V_{i,K,\sigma}^k\big).
\end{equation}
By assumption, $\text{D}_{K,\sigma}(u_i^k)\ge 0$.
If $\V_{i,K\sigma}^k\ge 0$, we have
$-u_{i,\sigma}^k\V_{i,K,\sigma}^k = -(u_{i,K})^+\V_{i,K,\sigma}=0$ and if
$\V_{i,K,\sigma}^k<0$, it follows that $-u_{i,\sigma}^k\V_{i,K,\sigma}^k
= -(u_{i,K,\sigma}^k)^+\V_{i,K,\sigma}^k\ge 0$. Hence, the right-hand side of
\eqref{3.aux} nonnegative. However, the left-hand side is negative, which is a 
contradiction. We infer that $u_{i,K}^k\ge 0$ and consequently,
$u_{i,L}^k\ge 0$ for all $L\in\T$. 
When the initial data are positive, similar arguments
show the positivity of $u_{i,L}^k$ for $L\in\T$.
\end{proof}

We are able to show the nonnegativity of $u_{0,K}^k=1-\sum_{i=1}^n u_{i,K}^k$ 
only if the diffusion
coefficients are the same. The reason is that we derive an equation for $u_{0,K}^k$
by summing \eqref{2.equ} for $i=1,\ldots,n$, and this gives an equation for
$u_{0,K}^k$ only if $D_i=D$ for all $i=1,\ldots,n$.

\begin{lemma}[Nonnegativity of $u_0^k$]\label{lem.u0}
Let Assumption (A2) hold and let $(u,\Phi)$ be a solution to 
\eqref{2.equ}, \eqref{2.eqF}, \eqref{2.poi}, and \eqref{3.trunc}.
Then $u_{0,K}^k\ge 0$ for all $K\in\T$, $k\in\{1,\ldots,N\}$. 
If $u_0^I>0$ and $\overline{u}_i>0$ then also 
$u_{0,K}^k>0$ for all $K\in\T$, $k\in\{1,\ldots,N\}$.
\end{lemma}

\begin{proof}
Again, we proceed by induction. The case $k=0$ follows from the assumptions.
Assume that $u_{0,L}^{k-1}\ge 0$ for all $L\in\T$. Then let $u_{0,K}^k
= \min\{u_{0,L}^k:L\in\T\} $ for some $K\in\T$ and assume that $u_{0,K}^k<0$. Summing equations \eqref{2.equ}
from $i=1,\ldots,n$, we obtain
\begin{align}
  \m(K)\frac{u_{0,K}^k-u_{0,K}^{k-1}}{\dt}
	&= D\sum_{\sigma\in\E_K}\tau_\sigma\Big(u_{0,\sigma}^k\text{D}_{K,\sigma}(u_0^k)
	+ \sum_{i=1}^n u_{i,\sigma}^k\big(\text{D}_{K,\sigma}(u_0^k) 
	- \beta z_i\widehat{u}_{0,\sigma,i}^k\text{D}_{K,\sigma}(\Phi^k)\big)\Big) \nonumber\\
	&\ge -D\sum_{\sigma\in\E_K}\tau_\sigma 
	\sum_{i=1}^n\beta z_i\widehat{u}_{0,\sigma,i}^k\text{D}_{K,\sigma}(\Phi^k),
	\label{3.aux2}
\end{align}
since $u_{0,\sigma}^k\ge 0$ and $u_{i,\sigma}^k\ge 0$ by construction
and $\text{D}_{K,\sigma}(u_0^k)\ge 0$ because of the minimality property
of $u_{0,K}^k$. The remaining expression is nonnegative:
$$
  -\widehat u^k_{0,\sigma_i}z_i\text{D}_{K,\sigma}(\Phi^k) = \begin{cases}
  -(u^k_{0,K})^+z_i\text{D}_{K,\sigma}(\Phi^k)=0 \quad 
	& \text{if }z_i\text{D}_{K,\sigma}(\Phi^k)\ge 0, \\
  -(u^k_{0,L})^+z_i\text{D}_{K,\sigma}(\Phi^k)\ge 0 
	& \text{if }z_i\text{D}_{K,\sigma}(\Phi^k)< 0.
 \end{cases}
$$
However, the left-hand side of \eqref{3.aux2} is negative, by induction hypothesis,
which gives a contradiction.
\end{proof}

Lemmas \ref{lem.ui} and \ref{lem.u0} imply that we may remove the truncation 
in \eqref{3.trunc}. Moreover, by definition, we have $1-\sum_{i=1}^n u_{i,K}^k
=u_{0,K}^k\ge 0$ such that $u_K^k=(u_{1,K}^k,\ldots,u_{n,K}^k)\in\overline{\D}$
or, if the initial and boundary data are positive, $u_K^k\in\D$.

\begin{proposition}[Existence for the numerical scheme]\label{prop.ex}
Let Assumption (A2) hold. 
Then scheme \eqref{2.equ}-\eqref{2.poi} has a solution
$(u,\Phi)$ which satisfies $u_K^k\in\overline{\D}$ for all $K\in\T$ and $k\in\N$.
\end{proposition}

\begin{proof}
We argue by induction. For $k=0$, we have $u_K^0\in\overline{\D}$ by assumption.
The function $\Phi^0$ is uniquely determined by scheme \eqref{2.poi}, as this is a 
linear system of equations with positive definite matrix. Assume the existence
of a solution $(u^{k-1},\Phi^{k-1})$ with $u_K^{k-1}\in\overline{\D}$.
Let $m\in\N$ be the product of the number of species $n$ and the number of
cells $K\in\T$. 
For given $K\in\T$ and $i=1,\ldots,n$, we define the function
$F_{i,K}:\R^m\times[0,1]\to\R$ by
\begin{align*}
  F_{i,K}(u,\rho) &= \m(K)\frac{u_{i,K}-u_{i,K}^{k-1}}{\dt} \\
	&\phantom{xx}{}- \rho D\sum_{\sigma\in\E_K}\tau_\sigma\Big(u_{0,\sigma}
	\text{D}_{K,\sigma}(u_i) - u_{i,\sigma}\big(\text{D}_{K,\sigma}(u_0)
	- \widehat{u}_{0,\sigma,i}\beta z_i\text{D}_{K,\sigma}(\Phi)\big)\Big).
\end{align*}
where $u_{0,K}$, $u_{i,\sigma}$, $u_{0,\sigma}$, and $\widehat u_{0,\sigma_i}$
are defined in \eqref{3.trunc}, and $\Phi$ is uniquely determined 
by \eqref{2.poi}.
Let $F=(F_{i,K})_{i=1,\ldots,n,\,K\in\T}$. Then $F:\R^m\times[0,1]\to\R^m$
is a continuous function. We wish to apply the fixed-point theorem of 
\cite[Theorem 5.1]{EGH08}. For this, we need to verify three assumptions:
\begin{itemize}
\item The function $u\mapsto F_{i,K}(u,0) = \m(K)(u_{i,K}-u_{i,K}^{k-1})/\dt$ is affine.
\item We have proved above that any solution to $F(u,1)=0$ satisfies $u\in\D$
or $\|u\|_\infty<2$.
A similar proof shows that any solution to $F(u,\rho)=0$ with $\rho\in(0,1)$
satisfies $\|u\|_\infty<2$, too.
\item The equation $F(u,0)=0$ has the unique solution $u=u^{k-1}$ and consequently,
$\|u\|_\infty=\|u^{k-1}\|_\infty<2$. 
\end{itemize}
We infer the existence of a solution $u^k$ to $F(u^k,1)=0$ satisfying
$\|u^k\|_\infty<2$. In fact, by Lemmas \ref{lem.ui} and \ref{lem.u0},
we find that $u^k\in\overline\D$. Hence, $u^k$ solves the original scheme 
\eqref{2.equ}-\eqref{2.poi}.
\end{proof}


\subsection{Uniqueness of solutions}\label{sec.uni}

The proof of Theorem \ref{thm.ex} is completed when we show the uniqueness
of solutions to scheme \eqref{2.equ}-\eqref{2.poi} under the additional conditions (A1) and
(A3). Recall that in this case, the scheme is given by 
\eqref{2.simpl1}-\eqref{2.simpl2},

{\em Step 1: uniqueness for $u_0$.} If $k=0$, the solution is uniquely determined
by the initial condition. Assume that $u_0^{k-1}$ is given. Thanks to Assumptions
(A2)-(A3), the sum of \eqref{2.simpl1}-\eqref{2.simpl2} for $i=1,\ldots,n$ 
gives an equation for
$u_0^k=1-\sum_{i=1}^n u_i^k$ (in the following, we omit the superindices $k$):
\begin{align*}
  \m(K)\frac{u_{0,K}-u_{0,K}^{k-1}}{\dt}
	&= -D\sum_{\sigma\in\E_{K,\rm int}}\tau_\sigma(u_{0,K}-u_{0,L})
	\bigg(u_{0,\sigma} + \sum_{i=1}^n u_{i,\sigma}\bigg) \\
	&= -D\sum_{\sigma\in\E_{K,\rm int}}\tau_\sigma(u_{0,K}-u_{0,L})
	\big(1 + |u_{0,K}-u_{0,L}|\big),
\end{align*}
where we used \eqref{2.sum} in the last step. 

Let $u_0$ and $v_0$ be two solutions to the previous equation and set $w_0:=u_0-v_0$. 
Then $w_0$ solves
\begin{align*}
  0 &= \m(K)\frac{w_{0,K}}{\dt} + D\sum_{\sigma\in\E_{K,\rm int}}
	\tau_\sigma(w_{0,K}-w_{0,L}) \\
	&\phantom{xx}{}+ D\sum_{\sigma\in\E_{K,\rm int}}\tau_\sigma
	\big((u_{0,K}-u_{0,L})|u_{0,K}-u_{0,L}| - (v_{0,K}-v_{0,L})|v_{0,K}-v_{0,L}|\big).
\end{align*}
We multiply this equation by $w_{0,K}/D$, sum over $K\in\T$, and use discrete
integration by parts \eqref{2.ibp}:
\begin{align*}
  0 &= \sum_{K\in\T}\frac{\m(K)}{D}\frac{w_{0,K}^2}{\dt}
	+ \sum_{\sigma=K|L\in\E_{\rm int}}\tau_\sigma(w_{0,K}-w_{0,L})^2 \\
	&\phantom{xx}{} + \sum_{\sigma=K|L\in\E_{\rm int}}\tau_\sigma
	\big((u_{0,K}-u_{0,L})|u_{0,K}-u_{0,L}| - (v_{0,K}-v_{0,L})|v_{0,K}-v_{0,L}|\big)
	(w_{0,K}-w_{0,L}).
\end{align*}
The first two terms on the right-hand side are clearly nonnegative.
We infer from the elementary inequality
$(y|y|-z|z|)(y-z)\ge 0$ for $y$, $z\in\R$, which is a consequence of the
monotonicity of $z\mapsto z|z|$, that the third term is nonnegative, too.
Consequently, the three terms must vanish and this implies that $w_{0,K}=0$
for all $K\in\T$. This shows the uniqueness for $u_0$.

{\em Step 2: uniqueness for $u_i$.} 
Let $u_0$ be the uniquely determined solution from the previous step
and let $u^k=(u_1^k,\ldots,u_n^k)$ and $v^k=(v_1^k,\ldots,v_n^k)$ be two solutions
to \eqref{2.equ}. Similarly as in \cite{Gaj94}, we introduce the semimetric
\begin{align*}
  & d_\eps(u^k,v^k) = \sum_{K\in\T}\m(K)\sum_{i=1}^n H_1^\eps(u_{i,K}^k,v_{i,K}^k), 
	\quad\mbox{where} \\
	& H_1^\eps(a,b) = h_\eps(a)+h_\eps(b)
	- 2h_\eps\bigg(\frac{a+b}{2}\bigg)
\end{align*}
and $h_\eps(z) = (z+\eps)(\log(z+\eps)-1)+1$. The parameter $\eps>0$ is needed
since $u^k_{i,K}$ or $v^k_{i,K}$ may vanish and then the logarithm of $u^k_{i,K}$ or
$v^k_{i,K}$ may be undefined. 
The objective is to verify that $\lim_{\eps\to 0}d_\eps(u^k,v^k)=0$ by estimating the
discrete time derivative of the semimetric, implying that $u^k=v^k$.

First, we write
$$
  d_\eps(u^k,v^k) - d_\eps(u^{k-1},v^{k-1})
	= \sum_{K\in\T}\m(K)\sum_{i=1}^n\big(H_1^\eps(u_{i,K}^k,v_{i,K}^k)
	- H_1^\eps(u_{i,K}^{k-1},v_{i,K}^{k-1})\big).
$$
The function $H_1^\eps$ is convex since
$$
  D^2 H_1^\eps(a,b) = \frac{1}{(a+\eps)(b+\eps)(a+b+2\eps)}
	\begin{pmatrix}
	(b+\eps)^2 & -(a+\eps)(b+\eps) \\
	-(a+\eps)(b+\eps) & (a+\eps)^2
	\end{pmatrix}.
$$
Therefore, a Taylor expansion of $H_1^\eps$ around $(u_{i,K}^k,v_{i,K}^k)$ 
leads to
\begin{align*}
  \frac{1}{\dt} & \big(d_\eps(u^k,v^k) - d_\eps(u^{k-1},v^{k-1})\big) \\
	&\le \sum_{K\in\T}\frac{\m(K)}{\dt}\sum_{i=1}^n\bigg\{DH_1^\eps(u_{i,K}^k,v_{i,K}^k)
	\bigg(\begin{pmatrix} u^{k}_{i,K} \\ v^{k}_{i,K} \end{pmatrix} 
	- \begin{pmatrix} u^{k-1}_{i,K} \\ v^{k-1}_{i,K} \end{pmatrix}\bigg)\bigg\} \\
	&= \sum_{i=1}^n\sum_{K\in\T}\m(K)\frac{u_{i,K}^k-u_{i,K}^{k-1}}{\dt}
	\bigg(h_\eps'(u_{i,K}^k) - h_\eps'\bigg(\frac{u_{i,K}^k+v_{i,K}^k}{2}\bigg)\bigg) \\
	&\phantom{xx}{}+ \sum_{i=1}^n\sum_{K\in\T}\m(K)\frac{v_{i,K}^k-v_{i,K}^{k-1}}{\dt}
	\bigg(h_\eps'(v_{i,K}^k) - h_\eps'\bigg(\frac{u_{i,K}^k+v_{i,K}^k}{2}\bigg)\bigg).
\end{align*}
We insert the scheme \eqref{2.simpl1}-\eqref{2.simpl2} and use discrete integration by parts:
$$
  \frac{1}{\dt}\big(d_\eps(u^k,v^k) - d_\eps(u^{k-1},v^{k-1})\big) 
  \le S_1^k + S_2^k + \eps S_3^k,
$$
where
\begin{align*}
  S_1^k &= -D\sum_{i=1}^n\sum_{\sigma=K|L\in\E_{\rm int}}\tau_\sigma u_{0,\sigma}^k\bigg\{
	\big(u_{i,K}^k-u_{i,L}^k\big)\big(\log(u_{i,K}^{k}+\eps)
	-\log (u_{i,L}^{k}+\eps)\big) \\
	&\phantom{xx}{}
	+ \big(v_{i,K}^k-v_{i,L}^k\big)\big(\log (v_{i,K}^{k}+\eps)
	-\log (v_{i,L}^{k}+\eps)\big) \\
	&\phantom{xx}{}
	- 2\bigg(\frac{u_{i,K}^k+v_{i,K}^k}{2}-\frac{u_{i,L}^k+v_{i,L}^k}{2}\bigg)
	\bigg(\log\bigg(\frac{u_{i,K}^{k}+v_{i,K}^{k}}{2}+\eps\bigg)
	- \log\bigg(\frac{u_{i,L}^{k}+v_{i,L}^{k}}{2}+\eps\bigg)\bigg)\bigg\}, \\
	S_2^k &= D\sum_{i=1}^n\sum_{\sigma=K|L\in\E_{\rm int}}\tau_\sigma(u_{0,K}^k-u_{0,L}^k)
	\bigg\{(u_{i,\sigma}^{k}+\eps)\big(\log (u_{i,K}^{k}+\eps)
	-\log (u_{i,L}^{k}+\eps)\big) \\
	&\phantom{xx}{}
	+ (v_{i,\sigma}^{k}+\eps)\big(\log (v_{i,K}^{k}+\eps)
	-\log (v_{i,L}^{k}+\eps)\big) \\
	&\phantom{xx}{}
	- 2\bigg(\frac{u_{i,\sigma}^{k}+v_{i,\sigma}^{k}}{2}+\eps\bigg)\bigg(
	\log\bigg(\frac{u_{i,K}^{k}+v_{i,K}^{k}}{2}+\eps\bigg) 
	- \log\bigg(\frac{u_{i,L}^{k}+v_{i,L}^{k}}{2}+\eps\bigg)\bigg)\bigg\}, \\
	S_3^k &= -D\sum_{i=1}^n\sum_{\sigma=K|L\in\E_{\rm int}}\tau_\sigma
	(u_{0,K}^k-u_{0,L}^k)
	\bigg\{\big(\log (u_{i,K}^{k}+\eps)-\log (u_{i,L}^{k}+\eps)\big) \\
	&\phantom{xx}{}
	+ \big(\log (v_{i,K}^{k}+\eps)-\log (v_{i,L}^{k}+\eps)\big) \\
	&\phantom{xx}{}
	-2 \bigg(\log\bigg(\frac{u_{i,K}^{k}+v_{i,K}^{k}}{2}+\eps\bigg)
	- \log\bigg(\frac{u_{i,L}^{k}+v_{i,L}^{k}}{2}+\eps\bigg)\bigg)\bigg\}
\end{align*}

We claim that $S_1^k\le 0$ and $S_2^k\le 0$. Indeed, with the definition
$H_2^\eps(a,b)=(a-b)(\log(a+\eps)-\log(b+\eps))$, we can reformulate $S_1^k$ as
\begin{align*}
  S_1^k &= -D\sum_{i=1}^n\sum_{\sigma=K|L\in\E_{\rm int}}\tau_\sigma u_{0,\sigma}^k
	\bigg\{H_2^\eps\big(u_{i,K}^k,u_{i,L}^k\big) + H_2^\eps\big(v_{i,K}^k,v_{i,L}^k\big) \\
	&\phantom{xx}{}
	- 2H_2^\eps\bigg(\frac{u_{i,K}^k+v^k_{i,K}}{2},\frac{u_{i,L}^k+v_{i,L}^k}{2}\bigg)
	\bigg\}.
\end{align*}
The Hessian of $H_2^\eps$,
$$
  D^2H_2^\eps(a,b) = \begin{pmatrix}
	\frac{a+b+2\eps}{(a+\eps)^2} & -\frac{a+b+2\eps}{(a+\eps)(b+\eps)} \\
	-\frac{a+b+2\eps}{(a+\eps)(b+\eps)} & \frac{a+b+2\eps}{(b+\eps)^2}
	\end{pmatrix},
$$
is positive semidefinite. Therefore, performing a Taylor expansion up to second
order, we see that $S_1^k\le 0$. 

Next, we show that $S_2^k\le 0$. For this, we assume without loss of generality
for some fixed $\sigma=K|L$ that $u_{0,K}^k\le u_{0,L}^k$. By definition of the scheme,
$u_{i,\sigma}^k=u_{i,K}^k$ and $v_{i,\sigma}^k=v_{i,K}^k$. 
Set $H_3^\eps(a,b)=(a+\eps)(\log(a+\eps)-\log(b+\eps))$. The term in the curly  
bracket in $S_2^k$ then takes the form
\begin{equation}\label{3.H3}
  \big(u_{0,K}^k-u_{0,L}^k\big)\bigg\{H_3^\eps\big(u_{i,K}^k,u_{i,L}^k\big)
	+ H_3^\eps\big(v_{i,K}^k,v_{i,L}^k\big)
	- 2H_3^\eps\bigg(\frac{u_{i,K}^k+v_{i,K}^k}{2},\frac{u_{i,L}^k+v_{i,L}^k}{2}\bigg)
	\bigg\}.
\end{equation}
The Hessian of $H_3^\eps$,
$$
  D^2H_3^\eps(a,b) = \begin{pmatrix}
	\frac{1}{a+\eps} & -\frac{1}{b+\eps} \\
	-\frac{1}{b+\eps} & \frac{a+\eps}{(b+\eps)^2}
	\end{pmatrix},
$$
is also positive semidefinite, showing that \eqref{3.H3} is nonpositive
as $u_{0,K}^k-u_{0,L}^k\le 0$.
If $u_{0,K}^k>u_{0,L}^k$, both factors of the product \eqref{3.H3} change their
sign, so that we arrive at the same conclusion. Hence, $S_2^k\le 0$. We conclude that
$$
  d_\eps(u^k,v^k) - d_\eps(u^{k-1},v^{k-1}) \le \eps \dt S_3^k.
$$
Since $d_\eps(u^0,v^0)=0$, we find after resolving the recursion that
$$
  d_\eps(u^k,v^k) \le \eps\dt\sum_{\ell=1}^k S_3^\ell.
$$
As the densities $u_{i,K}^\ell$ are nonnegative and bounded by 1 for all $K\in\T$, for all $\ell\geq 0$ and for all $1\leq i\leq n$, it is clear that $\sum_{\ell=1}^k\eps S_3^\ell \to 0$ when $\eps\to 0$.
Then, we may perform the limit
$\eps\to 0$ in the previous inequality yielding $d_\eps(u^k,v^k)\to 0$.
A Taylor expansion as in \cite[end of Section 6]{ZaJu17}
shows that $d_\eps(u^k,v^k)\ge\frac18\sum_{K\in\T}
\m(K)\sum_{i=1}^n(u^k_{i,K}-v^k_{i,K})^2$. We infer that $u^k=v^k$, finishing the proof.


\section{Discrete entropy inequality and uniform estimates}\label{sec.est}

\subsection{Discrete entropy inequality}\label{sec.dei}

First, we prove \eqref{1.epi}.

\begin{proof}[Proof of Theorem \ref{thm.ent}]
The idea is to multiply \eqref{2.equ} by $\log(u_{i,K}^{k,\eps}/u_{0,K}^{k,\eps})$,
where $u_{i,K}^{k,\eps}:=u_{i,K}^k+\eps$ for $i=0,\ldots,n$. 
The regularization is necessary
to avoid issues when the concentrations vanish. After this multiplication, we sum
the equations over $i=1,\ldots,n$ and $K\in\T$ and use discrete integration 
by parts to obtain
\begin{align}
  0 &= \sum_{K\in\T}\frac{\m(K)}{\dt D}\sum_{i=1}^n(u_{i,K}^k-u_{i,K}^{k-1})
	\log\frac{u_{i,K}^{k,\eps}}{u_{0,K}^{k,\eps}} \nonumber \\
	&\phantom{xx}{}+ \sum_{\sigma=K|L\in\E_{\rm int}}\tau_\sigma
	\Big(u_{0,\sigma}^k\big(u_{i,K}^k-u_{i,L}^k\big)
	- u_{i,\sigma}^k\big(u_{0,K}^k-u_{0,L}^k\big)\Big)
	\bigg(\log\frac{u_{i,K}^{k,\eps}}{u_{0,K}^{k,\eps}}
	- \log\frac{u_{i,L}^{k,\eps}}{u_{0,L}^{k,\eps}}\bigg) \label{3.AB} \\
	&= A_0 + \sum_{\sigma=K|L\in\E_{\rm int}}\tau_\sigma(A_1+A_2+B_1+B_2),
	\nonumber
\end{align}
where
\begin{align*}
  A_0 &= \sum_{K\in\T}\frac{\m(K)}{\dt D}\sum_{i=0}^n
	(u_{i,K}^{k,\eps}-u_{i,K}^{k-1,\eps})
	\log u_{i,K}^{k,\eps}, \\
	A_1 &= \sum_{i=1}^n u_{0,\sigma}^k\big(u_{i,K}^{k,\eps}-u_{i,L}^{k,\eps}\big)
	\big(\log u_{i,K}^{k,\eps} - \log u_{i,L}^{k,\eps}\big), \\
	A_2 &= -\sum_{i=1}^n u_{0,\sigma}^k\big(u_{i,K}^k-u_{i,L}^k\big)
	\big(\log u_{0,K}^{k,\eps} - \log u_{0,L}^{k,\eps}\big), \\
	B_1 &= -\sum_{i=1}^n u_{i,\sigma}^k\big(u_{0,K}^k-u_{0,L}^k\big)
	\big(\log u_{i,K}^{k,\eps} - \log u_{i,L}^{k,\eps}\big), \\
  B_2 &= \sum_{i=1}^n u_{i,\sigma}^k\big(u_{0,K}^k-u_{0,L}^k\big)
	\big(\log u_{0,K}^{k,\eps} - \log u_{0,L}^{k,\eps}\big).
\end{align*}
The convexity of $h(z)=z(\log z-1)+1$ implies the inequality
$h(u)-h(v)\le h'(u)(u-v)$ for all $u$, $v\in\R$. Consequently,
$$
  A_0 \ge \sum_{K\in\T}\frac{\m(K)}{\dt D}\sum_{i=0}^n
	\big(u_{i,K}^{k,\eps}(\log u_{i,K}^{k,\eps}-1)
	- u_{i,K}^{k-1,\eps}(\log u_{i,K}^{k-1,\eps}-1)\big).
$$

In order to estimate the remaining terms, we recall two elementary inequalities.
Let $y$, $z>0$. Then, by the Cauchy-Schwarz inequality,
\begin{equation}\label{3.ineq1}
  \big(\sqrt{y}-\sqrt{z}\big)^2 
	= \bigg(\int_z^y \frac{ds}{2\sqrt{s}}\bigg)^2
	\le \int_z^y \frac{ds}{4}\int_z^y\frac{ds}{s} = \frac14(y-z)(\log y-\log z),
\end{equation}
and by the concavity of the logarithm,
\begin{equation}\label{3.ineq2}
  y(\log y-\log z) \ge y-z \ge z(\log y-\log z).
\end{equation}
Inequality \eqref{3.ineq1} shows that
$$
  A_1 \ge 4\sum_{i=1}^n u_{0,\sigma}^k\big((u_{i,K}^{k,\eps})^{1/2}
	- (u_{i,K}^{k,\eps})^{1/2}\big).
$$
We use the definition of $u_{0,K}^k=1-\sum_{i=1}^n u_{i,K}^k$ in $A_2$ to find that
$$
  A_2 = u_{0,\sigma}^k\big(u_{0,K}^k-u_{0,L}^k\big)
	\big(\log u_{0,K}^{k,\eps}-\log u_{0,L}^{k,\eps}\big).
$$

We rewrite $B_1$ by using the abbreviation $u_{i,\sigma}^{k,\eps}=u_{i,\sigma}^k+\eps$:
\begin{align*}
  B_1 &= -\sum_{i=1}^n u_{i,\sigma}^{k,\eps}\big(u_{0,K}^k-u_{0,L}^k\big)
	\big(\log u_{i,K}^{k,\eps}-\log u_{i,L}^{k,\eps}\big) \\
	&\phantom{xx}{}+ \eps\sum_{i=1}^n\big(u_{0,K}^k-u_{0,L}^k\big)
	\big(\log u_{i,K}^{k,\eps}-\log u_{i,L}^{k,\eps}\big) \\
	&=: B_{11} + \eps B_{12}.
\end{align*}
We apply inequality \eqref{3.ineq2} to $B_{11}$. Indeed, if 
$u_{0,K}^k\le u_{0,L}^k$, we have $u_{i,\sigma}^k=u_{i,K}^k$
and we use the first inequality in \eqref{3.ineq2}. If $u_{0,K}^k>u_{0,L}^k$
then $u_{i,\sigma}^k=u_{i,L}^k$ and we employ the second inequality in \eqref{3.ineq2}.
In both cases, it follows that
\begin{align*}
  B_{11} &\ge -\sum_{i=1}^n\big(u_{0,K}^k-u_{0,L}^k\big)
	\big(u_{i,k}^{k,\eps}-u_{i,L}^{k,\eps}\big) \\
	&= -\big(u_{0,K}^k-u_{0,L}^k\big)\sum_{i=1}^n\big(u_{i,k}^{k,\eps}-u_{i,L}^{k,\eps}\big)
	= \big(u_{0,K}^k-u_{0,L}^k\big)^2.
\end{align*}

Finally, we consider $B_2$. In view of Assumption (A3), equation \eqref{2.sum} gives
\begin{equation}\label{4.sum}
  \sum_{i=1}^n u_{i,\sigma}^k = 1-\min\{u_{0,K}^k,u_{0,L}^k\}\ge 1-u_{0,\sigma}^k,
\end{equation} 
and therefore, by \eqref{3.ineq1},
\begin{align*}
  B_2 &\ge \big(1-u_{0,\sigma}^k\big)\big(u_{0,K}^{k,\eps}-u_{0,L}^{k,\eps}\big)
	\big(\log u_{0,K}^{k,\eps}-\log u_{0,L}^{k,\eps}\big) \\
	&\ge 4\big((u_{0,K}^{k,\eps})^{1/2} - (u_{0,L}^{k,\eps})^{1/2}\big)^2
	- u_{0,\sigma}^k\big(u_{0,K}^{k}-u_{0,L}^{k}\big)
	\big(\log u_{0,K}^{k,\eps}-\log u_{0,L}^{k,\eps}\big).
\end{align*}
The last expression cancels with $A_2$ such that
$$
  A_2+B_2 \ge 4\big((u_{0,K}^{k,\eps})^{1/2} - (u_{0,L}^{k,\eps})^{1/2}\big)^2.
$$
Putting together the estimates for $A_0$, $A_1$, $B_1$, and $A_2+B_2$, we deduce from
\eqref{3.AB} that
\begin{align*}
  \sum_{K\in\T} & \frac{\m(K)}{\dt}\sum_{i=0}^n u_{i,K}^{k,\eps}(\log u_{i,K}^{k,\eps}-1)
	- \sum_{K\in\T}\frac{\m(K)}{\dt}\sum_{i=1}^n
	u_{i,K}^{k-1,\eps}(\log u_{i,K}^{k-1,\eps}-1) \\
	&\phantom{xx}{}+ D\sum_{\sigma=K|L\in\E_{\rm int}}\tau_\sigma\bigg\{
	4\sum_{i=1}^n u_{0,\sigma}\big((u_{i,K}^{k,\eps})^{1/2}-(u_{i,L}^{k,\eps})^{1/2}\big)^2
	\\
	&\phantom{xx}{}
	+ 4\big((u_{0,K}^{k,\eps})^{1/2}-(u_{0,L}^{k,\eps})^{1/2}\big)^2
	+ \big(u_{0,K}^k-u_{0,L}^k\big)^2\bigg\} \\
	&\le -\eps D\big(u_{0,K}^k-u_{0,L}^k\big)\sum_{i=1}^n
	\big(\log u_{i,K}^{k,\eps}-\log u_{i,L}^{k,\eps}\big).
\end{align*}
Since the right-hand side converges to zero as $\eps\to 0$, we infer that 
\eqref{1.epi} holds.
\end{proof}


\subsection{A priori estimates}

For the proof of the convergence result, we need estimates uniform in the
mesh size $h(\T)$ and time step $\dt$. The scheme provides uniform
$L^\infty$ bounds. Further bounds are derived from the discrete entropy
inequality of Theorem \ref{thm.ent}. We introduce the discrete time derivative 
for functions $v\in\mathcal{H}_{\T,\dt}$ by
\begin{equation}\label{4.dtime}
  \pa_t^\dt v^k = \frac{v^k-v^{k-1}}{\dt}, \quad k=1,\ldots,N.
\end{equation}

\begin{lemma}[A priori estimates]\label{lem.est}
Let (H1)-(H4) and (A1)-(A3) hold.
The solution $u$ to scheme \eqref{2.simpl1}-\eqref{2.simpl2} satisfies the following
uniform estimates:
\begin{align}
  \|u_0^{1/2}\|_{1,\T,\dt} + \|u_0^{1/2}u_i\|_{1,\T,\dt} &\le C, \quad i=1,\ldots,n, 
	\label{4.est1} \\
	\sum_{k=1}^N\dt\|\pa_t^\dt u^k_i\|_{-1,\T}^2 &\le C, \quad
	i=0,\ldots,n, \label{4.est2}
\end{align}
where the constant $C>0$ is independent of the mesh $\T$ and time step size $\dt$.
\end{lemma}

\begin{proof}
We claim that estimates \eqref{4.est1} follow from the discrete entropy inequality
\eqref{1.epi}. Indeed, we sum \eqref{1.epi} over $k=1,\ldots,N$ to obtain
\begin{align*}
  H^N &+ D\sum_{k=1}^N\dt\sum_{\sigma=K|L\in\E_{\rm int}}\tau_\sigma
	\bigg(4\sum_{i=1}^n u_{0,\sigma}^k\big((u_{i,K}^k)^{1/2}-(u_{i,L}^k)^{1/2}\big)^2 \\
  &{}+ 4\big((u_{0,K}^k)^{1/2}-(u_{0,L}^k)^{1/2}\big)^2 
	+ \big(u_{0,K}^k-u_{0,L}^k\big)^2\bigg) \le H^0.
\end{align*}
Since the entropy at time $t=0$ is bounded independently of the discretization,
we infer immediately the bound for $u_0^{1/2}$ in ${\mathcal H}_{\T,\dt}$. 
For the bound on $u_0^{1/2}u_i$ in ${\mathcal H}_{\T,\dt}$, we observe that
\begin{align*}
  &(u_{0,K}^k)^{1/2}u_{i,K}^k - (u_{0,L}^k)^{1/2}u_{i,L}^k \\
	&\phantom{x}= u_{i,K}\big((u_{0,K}^k)^{1/2}-(u_{0,L}^k)^{1/2}\big)
	+ (u_{0,L}^k)^{1/2}\big((u_{i,K}^k)^{1/2}+(u_{i,L}^k)^{1/2}\big)
	\big((u_{i,K}^k)^{1/2}-(u_{i,L}^k)^{1/2}\big).
\end{align*}
Therefore, together with the $L^\infty$ bounds on $u_i$,
\begin{align*}
  &\sum_{\sigma=K|L\in\E_{\rm int}} \tau_\sigma
	\big((u_{0,K}^k)^{1/2}u_{i,K}^k - (u_{0,L}^k)^{1/2}u_{i,L}^k\big)^2 \\
	&\phantom{x}\le \sum_{\sigma=K|L\in\E_{\rm int}}\tau_\sigma
	\big((u_{0,K}^k)^{1/2}-(u_{0,L}^k)^{1/2}\big)^2
	+ 2\sum_{\sigma=K|L\in\E_{\rm int}}\tau_\sigma u_{0,\sigma}^k
	\big((u_{i,K}^k)^{1/2}-(u_{i,L}^k)^{1/2}\big)^2.
\end{align*}
Then, summing over $k=0,\ldots,N$ and using the estimates from the
entropy inequality, we achieve the bound on $u_0^{1/2}u_i$.

It remains to prove estimate \eqref{4.est2}. To this end, let $\phi\in{\mathcal H}_\T$
be such that $\|\phi\|_{1,\T}=1$ and let $k\in\{1,\ldots,N\}$ and $i\in\{1,\ldots,n\}$. 
We multiply the scheme \eqref{2.simpl1} by $\Phi_K$ and we sum over $K\in\T$. Using 
successively discrete integration by parts, the rewriting of the numerical fluxes 
\eqref{2.Fsqrt}, the Cauchy-Schwarz inequality, and the $L^\infty$ bounds on $u_i$, 
we compute
\begin{align*}
  &\sum_{K\in\T}\frac{\m(K)}{\dt}\big(u_{i,K}^k-u_{i,K}^{k-1}\big)\phi_K \\
	&= D\sum_{\sigma=K|L\in\E_{\rm int}}\tau_\sigma(u_{0,\sigma}^k)^{1/2}
	\big((u_{0,K}^k)^{1/2}u_{i,K}^k - (u_{0,L}^k)^{1/2}u_{i,L}^k\big)(\phi_K-\phi_L) \\
	&\phantom{xx}{}- D\sum_{\sigma=K|L\in\E_{\rm int}}\tau_\sigma
	\big((u_{0,K}^k)^{1/2}-(u_{0,L}^k)^{1/2}\big) \\
	&\phantom{xxxx}{}\times u_{i,\sigma}^k
	\bigg((u_{0,\sigma}^k)^{1/2}
	+ 2\frac{(u_{0,K}^k)^{1/2}+(u_{0,L}^k)^{1/2}}{2}\bigg)
	(\phi_K-\phi_L) \\
	&\le D\bigg(\sum_{\sigma=K|L\in\E_{\rm int}}\tau_\sigma
	\big((u_{0,K}^k)^{1/2}u_{i,K}^k - (u_{0,L}^k)^{1/2}u_{i,L}^k\big)^2\bigg)^{1/2}
	\bigg(\sum_{\sigma=K|L\in\E_{\rm int}}\tau_\sigma(\phi_K-\phi_L)^2\bigg)^{1/2} \\
	&\phantom{xx}{}+ 3D\bigg(\sum_{\sigma=K|L\in\E_{\rm int}}\tau_\sigma
	\big((u_{0,K}^k)^{1/2}-(u_{0,L}^k)^{1/2}\big)^2\bigg)^{1/2}
	\bigg(\sum_{\sigma=K|L\in\E_{\rm int}}\tau_\sigma(\phi_K-\phi_L)^2\bigg)^{1/2}.
\end{align*}
This shows that, for $i=1,\ldots,n$,
$$
  \sum_{k=1}^N\dt\bigg\|\frac{u_i^k-u_i^{k-1}}{\dt}\bigg\|_{-1,\T}^2
	\le 2D^2\sum_{k=1}^N\dt\Big(\big\|(u_{0}^k)^{1/2}u_i^k\big\|_{1,\T}^2
	+ 9\big\|(u_0^k)^{1/2}\big\|_{1,\T}^2\Big) \le C,
$$
as a consequence of \eqref{4.est1}. The estimate for 
$\dt^{-1}(u_0^k-u_0^{k-1}) = -\dt^{-1}\sum_{i=1}^n(u_i^k-u_i^{k-1})$
follows from those for $i=1,\ldots,n$, completing the proof.
\end{proof}


\section{Convergence of the scheme}\label{sec.conv}

In this section, we establish the convergence of the sequence of approximate
solutions, constructed in Theorem \ref{thm.ex}, to a weak solution to
\eqref{1.eq}, i.e., we prove Theorem \ref{thm.conv}.

\subsection{Compactness of the approximate solutions}

In order to achieve the convergence in the fluxes, we proceed as in \cite{CLP03}
by defining the approximate gradient on a dual mesh. For $\sigma=K|L\in\E_{\rm int}$,
we define the new cell $T_{KL}$ as the cell with the vertexes $x_K$, $x_L$
and those of $\sigma$. For $\sigma\in\E_{\rm ext}\cap\E_K$, we define $T_{K\sigma}$
as the cell with vertex $x_K$ and those of $\sigma$. Then $\Omega$ can be
decomposed as
$$
  \overline\Omega = \bigcup_{K\in\T}\bigg\{\bigg(\bigcup_{L\in \mathcal{N}_K}
	\overline{T}_{KL}\bigg)\cup
	\bigg(\bigcup_{\sigma\in\E_{{\rm ext},K}}\overline{T}_{K\sigma}\bigg)\bigg\},
$$
where $\mathcal{N}_K$ denotes the set of neighboring cells of $K$. 
The discrete gradient $\na_{\T,\dt}v$ on $\Omega_T:=\Omega\times(0,T)$ for
piecewise constant functions $v\in\mathcal{H}_{\T,\dt}$ is defined by
\begin{equation}\label{5.grad}
  \na_{\T,\dt}v(x,t) = \begin{cases}
  \displaystyle\frac{\m(\sigma)(v_L^k-v_K^k)}{\m(T_{KL})}\textbf{n}_{KL}
	\quad &\text{for }x\in T_{KL},\ t\in (t^k,t^{k+1}),\\
  0\quad &\text{for }x\in T_{K\sigma},\ t\in (t^k,t^{k+1}),
\end{cases}
\end{equation}
where $\mathbf{n}_{KL}$ denotes the unit normal on $\sigma=K|L$ oriented from $K$
to $L$. To simplify the notation, we set $\na_m:=\na_{\T_m,\dt_m}$. The solution
to the approximate scheme \eqref{2.simpl1}-\eqref{2.simpl2} is called
$u_{0,m},u_{1,m},\ldots,u_{n,m}$.

\begin{lemma}\label{lem.conv}
There exist functions $u_0\in L^\infty(\Omega_T)\cap L^2(0,T;H^1(\Omega))$
and $u_1,\ldots,u_n\in L^\infty(\Omega_T)$ such that, possibly for subsequences,
as $m\to\infty$,
\begin{align}
  u_{0,m}\to u_0, \quad u_{0,m}^{1/2}\to u_0^{1/2} 
	&\quad\mbox{strongly in }L^2(\Omega_T), \label{4.conv1} \\
	\na_m u_{0,m}\rightharpoonup \na u_0, \quad
	\na_m u_{0,m}^{1/2}\rightharpoonup \na u_0^{1/2} &\quad\mbox{weakly in }L^2(\Omega_T), 
	\label{4.conv2} \\
	u_{0,m}^{1/2}u_{i,m}\to u_0^{1/2}u_i &\quad\mbox{strongly in }L^2(\Omega_T), 
	\label{4.conv3} \\
	\na_m\big(u_{0,m}^{1/2}u_{i,m}\big)\rightharpoonup \na(u_0^{1/2}u_i)
	&\quad\mbox{weakly in }L^2(\Omega_T), \label{4.conv4}
\end{align}
where $i\in\{1,\ldots,n\}$.
\end{lemma}

\begin{proof}
First, we claim that $(u_{0,m})$ is uniformly bounded in $\mathcal{H}_{\T,\dt}$.
Indeed, by the $L^\infty$ bounds and estimate \eqref{4.est1},
\begin{align}
  \|u_{0,m}\|_{1,\T,\dt}^2
	&= \sum_{k=1}^N\dt\bigg(\sum_{\sigma=K|L\in\E_{\rm int}}\tau_\sigma(u^k_{0,K}-u^k_{0,L})^2
	+ \sum_{K\in\T}\m(K)(u_{0,K}^k)^2\bigg) \nonumber \\
	&= \sum_{k=1}^N\dt\bigg(\sum_{\sigma=K|L\in\E_{\rm int}}\tau_\sigma
	\big((u_{0,K}^k)^{1/2}+(u_{0,L}^k)^{1/2}\big)^2
	\big((u^k_{0,K})^{1/2}-(u^k_{0,L})^{1/2}\big)^2 \nonumber \\
	&\phantom{xx}{}+ \sum_{K\in\T}\m(K)(u_{0,K}^k)^2\bigg) \label{4.est3}\\
	&\le 4\|u_{0,m}\|_{L^\infty(\Omega_T)}\|u_{0,m}^{1/2}\|_{1,\T,\dt}^2  + \|u_{0,m}\|^2_{L^2(\Omega_T)}\le C.
	\nonumber 
\end{align}
By estimate \eqref{4.est2}, $(\pa_t^\dt u_{0,m})$ is uniformly bounded. Therefore,
by the discrete Aubin-Lions lemma
(see Lemma \ref{lem.aubin} in the appendix), we conclude the existence of a 
subsequence (not relabeled) such that the first convergence in \eqref{4.conv1} holds.
The strong convergence implies (up to a subsequence) that $u_{0,m}\to u_0$
pointwise in $\Omega_T$ and consequently $u_{0,m}^{1/2}\to u_0^{1/2}$ pointwise
in $\Omega_T$. Thus, together with the $L^\infty$ bound for $u_{0,m}^{1/2}$,
we infer the second convergence in \eqref{4.conv1}.

The convergences in \eqref{4.conv2} are a consequence
of the uniform estimates \eqref{4.est1} and \eqref{4.est3}
and the compactness result in \cite[proof of Theorem 10.3]{EGH00}. 
Applying the discrete Aubin-Lions lemma of ``degenerate'' type
(see Lemma \ref{lem.aubin2} in the appendix) to $y_m=u_{0,m}^{1/2}$ and
$z_m=u_{i,m}$ for fixed $i\in\{1,\ldots,n\}$, we deduce convergence \eqref{4.conv3}.
Finally, convergence \eqref{4.conv4} is a consequence of \eqref{4.conv3} and
the weak compactness of $(u_{0,m}^{1/2}u_{i,m})$, thanks to the uniform bound in
\eqref{4.est1}.
\end{proof}


\subsection{The limit $m\to\infty$}

We finish the proof of Theorem \ref{thm.conv} by verifying that the limit function
$u=(u_1,\ldots,u_n)$, as defined in Lemma \ref{lem.conv}, is a weak solution
in the sense of the theorem.

Let $\phi\in C_0^\infty(\overline{\Omega}\times[0,T))$ and let $m\in\N$ be large
enough such that $\operatorname{supp}\phi\subset\overline{\Omega}\times[0,(N_m-1)\dt_m)$
(recall that $T=N_m\dt_m$).
For the limit, we follow the strategy used, for instance, in \cite{CLP03} and
introduce the following notations:
\begin{align*}
	F_{10}(m)&=-\int_0^T\int_{\Omega}
	u_{i,m}\pa_t\phi dxdt-\int_{\Omega}u_{i,m}(0)\phi(0)dx, \\
	F_{20}(m)&=\int_0^T\int_{\Omega}
	u_{0,m}^{1/2}\na_m(u_{0,m}^{1/2}u_{i,m})\na\phi dxdt, \\
	F_{30}(m)&=3\int_0^T\int_{\Omega}
	u_{0,m}^{1/2}u_{i,m}\na_m(u_{0,m}^{1/2})\na\phi dxdt.
\end{align*}
The convergence results of Lemma \ref{lem.conv} show that, as $m\to\infty$,
\begin{align}
  F_{10}(m)+DF_{20}(m)-DF_{30}(m) &\to -\int_0^T\int_\Omega u_i\pa_t\phi dxdt
	- \int_\Omega u_i^0\phi(0)dx \label{4.sumF0} \\
	&\phantom{xx}{}
	+ D\int_0^T\int_\Omega\big(u_0^{1/2}\na(u_0^{1/2}u_i) - 3u_0^{1/2}u_i\na u_0^{1/2}
	\big)dxdt. \nonumber
\end{align}

Next, setting $\phi_K^k=\phi(x_K,t^k)$, we multiply scheme \eqref{2.simpl1}
by $\dt_m\phi_K^{k-1}$ and sum over $K\in\T_m$ and $k=1,\ldots,N_m$. Then
\begin{equation}\label{4.F123}
  F_1(m) + DF_2(m) - DF_3(m) = 0,
\end{equation}
where, omitting the subscript $m$ from now on to simplify the notation,
\begin{align*}
	F_1(m) &= \sum_{k=1}^{N}\sum_{K\in\T}\m(K)\big(u_{i,K}^k-u_{i,K}^{k-1}\big)
	\phi_K^{k-1}, \\
	F_2(m) &= \sum_{k=1}^{N}\dt\sum_{K\in\T}\sum_{\sigma\in\E_{K,\rm int}}\tau_\sigma 
	(u_{0,\sigma}^k)^{1/2}\big((u_{0,K}^k)^{1/2}u_{i,K}^k - (u_{0,L}^k)^{1/2}u_{i,L}^k
	\big)\phi_K^{k-1}, \\
	F_3(m) &= \sum_{k=1}^{N}\dt\sum_{K\in\T}\sum_{\sigma\in\E_{K,\rm int}}\tau_\sigma
	\big((u_{0,K}^k)^{1/2} - (u_{0,L}^k)^{1/2}\big) \\
	&\phantom{xx}{}\times u_{i,\sigma}^k\bigg((u_{0,\sigma}^k)^{1/2} 
	+ 2\frac{(u_{0,K}^k)^{1/2}+(u_{0,L}^k)^{1/2}}{2}\bigg)\phi_K^{k-1}.
\end{align*}

The aim is to show that $F_{i0}(m)-F_{i}(m)\to 0$ as $m\to\infty$ for $i=1,2,3$.
Then, because of \eqref{4.F123}, $F_{10}(m)+DF_{20}(m)-DF_{30}(m)\to 0$, which
finishes the proof. We start by verifying that $F_{10}(m)-F_1(m)\to 0$.
For this, we rewrite $F_1(m)$ and $F_{10}(m)$, using $\phi_K^N=0$:
\begin{align*}
	F_1(m) &= \sum_{k=1}^{N}\sum_{K\in\T}\m(K)u_{i,K}^k\big(\phi_K^{k-1}-\phi_K^k\big)
	-\sum_{K\in\T}\m(K)\phi_K^0u_{i,K}^0, \\
	&= -\sum_{k=1}^{N}\sum_{K\in\T}\int_{t^{k-1}}^{t^k}\int_K u_{i,K}^k\pa_t\phi(x_K,t)
	dxdt - \sum_{K\in\T}\int_Ku_{i,K}^0\phi(x_K,0)dx, \\
	F_{10}(m) &= -\sum_{k=1}^{N}\sum_{K\in\T}\int_{t^{k-1}}^{t^k}\int_K u_{i,K}^k
	\pa_t\phi(x,t)dxdt - \sum_{K\in\T}\int_Ku_{i,K}^0\phi(x,0)dx.
\end{align*}
In view of the regularity of $\phi$ and the uniform $L^\infty$ bound on $u_i$,
we find that 
$$
  |F_{10}(m)-F_{1}(m)| \le CT\m(\Omega)\|\phi\|_{C^2}
	h(\T_m) \to 0 \quad\mbox{as }m\to\infty.
$$

Using discrete integration by parts, the second integral becomes
\begin{align*}
  F_2(m) &= \sum_{k=1}^{N}\dt\sum_{\sigma=K|L\in\E_{\rm int}}\tau_\sigma (u_{0,\sigma}^k)^{1/2}
	\big((u_{0,K}^k)^{1/2}u_{i,K}^k - (u_{0,L}^k)^{1/2}u_{i,L}^k\big)
	\big(\phi_K^{k-1}-\phi_L^{k-1}\big) \\
	&= F_{21}(m)+F_{22}(m),
\end{align*}
where we have decomposed $(u_{0,\sigma}^k)^{1/2} = (u_{0,K}^k)^{1/2}
+ ((u_{0,\sigma}^k)^{1/2}-(u_{0,K}^k)^{1/2})$, i.e.
\begin{align*}
	F_{21}(m) &= \sum_{k=1}^{N}\dt\sum_{\sigma=K|L\in\E_{\rm int}}\tau_\sigma (u_{0,K}^k)^{1/2}
	\big((u_{0,K}^k)^{1/2}u_{i,K}^k - (u_{0,L}^k)^{1/2}u_{i,L}^k\big)
	\big(\phi_K^{k-1}-\phi_L^{k-1}\big), \\
	F_{22}(m) &= \sum_{k=1}^{N}\dt\sum_{\sigma=K|L\in\E_{\rm int}}\tau_\sigma 
	\big((u_{0,\sigma}^k)^{1/2} - (u_{0,K}^k)^{1/2}\big)
	\big((u_{0,K}^k)^{1/2}u_{i,K}^k - (u_{0,L}^k)^{1/2}u_{i,L}^k\big) \\
	&\phantom{xx}{}\times\big(\phi_K^{k-1}-\phi_L^{k-1}\big).
\end{align*}
Furthermore, we write $F_{20}(m)=G_1(m)+G_2(m)$, where
\begin{align*}
	G_1(m) &= \sum_{k=1}^{N}\sum_{\sigma=K|L\in\E_{\rm int}}\frac{\m(\sigma)}{\m(T_{KL})}
	(u_{0,K}^k)^{1/2}\big((u_{0,K}^k)^{1/2}u_{i,K}^k - (u_{0,L}^k)^{1/2}u_{i,L}^k\big) \\
	&\phantom{xx}{}\times
	\int_{t^{k-1}}^{t^k}\int_{T_{KL}}\na\phi(x,t)\cdot\textbf{n}_{K\sigma}dxdt, \\
	G_2(m) &= \sum_{k=1}^{N}\sum_{\sigma=K|L\in\E_{\rm int}}\frac{\m(\sigma)}{\m(T_{KL})}
	\big((u_{0,L}^k)^{1/2} - (u_{0,K}^k)^{1/2}\big)
	\big((u_{0,K}^k)^{1/2}u_{i,K}^k - (u_{0,L}^k)^{1/2}u_{i,L}^k\big) \\ 
	&\phantom{xx}{}\times\int_{t^{k-1}}^{t^k}\int_{T_{KL}\cap L}
	\na\phi(x,t)\cdot\textbf{n}_{K\sigma}dxdt.
\end{align*}

The aim is to show that $F_{21}(m)-G_1(m)\to 0$, $F_{22}(m)\to 0$, and $G_2(m)\to 0$.
This implies that
\begin{align*}
  |F_{20}(m)-F_2(m)| 
	&= \big|(G_1(m)+G_2(m))-(F_{21}(m)+F_{22}(m))\big| \\
	&\le |G_1-F_{21}| + |G_2| + |F_{22}| \to 0.
\end{align*}
First we notice that, due to the admissibility of the mesh and the regularity of $\phi$,
by taking the mean value over $T_{KL}$,
\begin{equation}\label{4.res1}
  \bigg|\int_{t^{k-1}}^{t^k}\bigg(\frac{\phi_K^{k-1}-\phi_L^{k-1}}{\dist_\sigma}
	- \frac{1}{\m(T_{KL})}\int_{T_{KL}}\na\phi(x,t)\cdot\textbf{n}_{K\sigma}\bigg)
	dt| \le C\dt h(\T),
\end{equation}
where the constant $C>0$ only depends on $\phi$. It yields 
\begin{align*}
|F_{21}(m)-G_1(m)| &\le Ch(\T)\sum_{k=1}^N\dt\sum_{\sigma=K|L\in\E_{\rm int}} \m(\sigma) \big|(u_{0,K}^k)^{1/2}u_{i,K} - (u_{0,L}^k)^{1/2}u_{i,L}\big|\\
%
%
	&\le Ch(\T)\|u_0^{1/2}u_i\|_{1,\T,\dt}(T\m(\Omega))^{1/2},
\end{align*}
where the last estimate follows from the Cauchy-Schwarz inequality.
This proves that $|F_{21}(m)-G_1(m)|\to 0$ as $m\to\infty$.

It remains to analyze the expressions $F_{22}(m)$ and $G_2(m)$. To this end,
we remark that $\dist_\sigma\le h(\T)$ and hence, together with
the regularity of $\phi$, and the Cauchy-Schwarz inequality,
\begin{align*}
  |F_{22}(m)| &\le \sum_{k=1}^N\dt\sum_{\sigma=K|L\in\E_{\rm int}}\tau_\sigma
	\big|(u_{0,\sigma}^k)^{1/2} - (u_{0,K}^k)^{1/2}\big|\,
	\big|(u_{0,K}^k)^{1/2}u_{i,K}^k - (u_{0,L}^k)^{1/2}u_{i,L}^k\big| \\
	&\phantom{xx}{}\times
	\frac{|\phi_K^{k-1}-\phi_L^{k-1}|}{\dist_\sigma}\dist_\sigma \\
	&\le Ch(\T)\|\phi\|_{C^1}\sum_{k=1}^N\dt\sum_{\sigma=K|L\in\E_{\rm int}}\tau_\sigma
	\big|(u_{0,\sigma}^k)^{1/2} - (u_{0,K}^k)^{1/2}\big| \\
	&\phantom{xx}{}\times
	\big|(u_{0,K}^k)^{1/2}u_{i,K}^k - (u_{0,L}^k)^{1/2}u_{i,L}^k\big| \\
  &\le Ch(\T)\|\phi\|_{C^1}\|u_0^{1/2}\|_{1,\T,\dt}\|u_0^{1/2}u_i\|_{1,\T,\dt}
	\le Ch(\T),
\end{align*}
The term $G_2(m)$ can be estimated in a similar way.

Finally, we need to show that $|F_{30}(m)-F_3(m)|\to 0$. The proof is completely
analogous to the previous arguments, since
\begin{align*}
	\bigg|3(u_{0,K}^k)^{1/2}&u_{i,K}^k - u_{i,\sigma}^k\bigg((u_{0,\sigma}^k)^{1/2}
	+ 2\frac{(u_{0,K}^k)^{1/2}+(u_{0,L}^k)^{1/2}}{2}\bigg) \bigg| \\
	&\le C\big((u_{0,\sigma}^k)^{1/2}|u_{i,K}^k-u_{i,L}^k| 
	+ |(u_{0,K}^k)^{1/2}-(u_{0,L}^k)^{1/2}|\big).
\end{align*}
Summarizing, we have proved that $|F_{i0}(m)-F_i(m)|\to 0$ for $i=1,2,3$,
and since $F_1(m)+DF_2(m)-DF_3(m)=0$, the convergence \eqref{4.sumF0}
shows that $u$ solves \eqref{1.weak}. This completes the proof of Theorem \ref{thm.conv}.


\section{Numerical experiments}\label{sec.num}

We present numerical simulations of a calcium-selective ion channel in two space
dimensions to illustrate the dynamical behavior of the ion transport model. 
Numerical simulations in one space dimension can be found in \cite{BSW12} 
for stationary solutions and in \cite{GeJu17} for transient solutions.
The channel is modeled as in \cite{GNE02}.
The selectivity of the channel is obtained by placing some 
confined oxygen ions (O$^{1/2-}$) inside the channel region. 
These ions contribute to the permanent charge density $f=-u_{\rm ox}/2$ in the 
Poisson equation, but also to the total sum of the concentrations. We consider
three further types of ions: calcium (Ca$^{2+}$, $u_1$), sodium (Na$^+$, $u_2$),
and chloride (Cl$^-$, $u_3$). While the concentrations of these ion species
satisfy the evolution equations \eqref{1.eq}, the oxygen concentration is constant
in time and given by the piecewise linear function
$$
  u_{\rm ox}(x,y) = u_{{\rm ox},\max}\times\begin{cases}
  1 \quad &\text{for }0.45 \le x \le 0.55, \\
  10(x-0.35) \quad &\text{for }0.35 \le x \le 0.45,\\
  10(0.65-x) \quad &\text{for }0.55 \le x \le 0.65, \\
  0 \quad &\text{else},
  \end{cases}
$$
where the scaled maximal oxygen concentration equals 
$u_{{\rm ox},\max}=(N_A/u_{\rm typ})\cdot 52\,$mol/L, where
$N_A\approx 6.022\cdot 10^{23}\,$mol$^{-1}$ is the Avogadro constant and
$u_{\rm typ}=3.7037\cdot 10^{25} L^{-1}$ the typical concentration
(taken from \cite[Table 1]{BSW12}). This gives $u_{{\rm ox},\max}\approx 0.84$.
The solvent concentration is computed according to $u_0=1-\sum_{i=1}^3 u_i-u_{\rm ox}$.
The physical parameters used in our simulations are taken from \cite[Table 1]{BSW12},
and the channel geometry is depicted in Figure \ref{fig.geom}. 
The boundary conditions are chosen as in \cite[Section~5]{BSW12}.

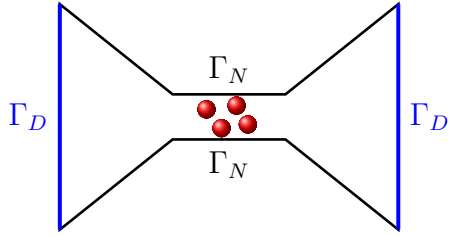
\begin{figure}
\begin{tikzpicture}
\draw[line width=1pt] (0,0) -- (0,3) node[midway, left, blue]{$\Gamma_D$} -- (1.5,1.8) -- (3,1.8) node[midway, above]{$\Gamma_N$} -- (4.5,3) -- (4.5,0) node[midway, right, blue]{$\Gamma_D$} -- (3,1.2) -- (1.5,1.2) node[midway, below]{$\Gamma_N$} -- (0,0)--cycle;
\draw[blue, line width=1.5pt] (0,0) -- (0,3);
\draw[blue, line width=1.5pt] (4.5,3) -- (4.5,0);
\shade[ball color=red] (2.15,1.35) circle (0.12);
\shade[ball color=red] (2.5,1.4) circle (0.12);
\shade[ball color=red] (2.35,1.65) circle (0.12);
\shade[ball color=red] (1.95,1.6) circle (0.12);
\end{tikzpicture}
\caption{Schematic picture of the ion channel $\Omega$ used for the simulations. 
Dirichlet boundary conditions are prescribed on $\Gamma_D$ (blue), homogeneous Neumann 
boundary conditions on $\Gamma_N$ (black). The red circles represent the confined 
$O^{1/2-}$ ions.}
\label{fig.geom}
\end{figure}

The simulations are performed with the full set of equations \eqref{1.eq}-\eqref{1.poi}
without assuming (A1)-(A3). The finite-volume scheme \eqref{2.equ}-\eqref{2.poi}
is implemented using
MATLAB, version R2015a. The nonlinear system defined by the implicit scheme is solved
with a full Newton method in the variables $u_0$, $u_1$, $u_2$, $u_3$, $\Phi$ 
for every time
step. The computations are done with a fixed time step size $\dt=10^{-3}$ until the
stationary state is approximately reached, i.e., until the discrete $L^2$ norm
between the solutions at two consecutive time steps is smaller than $10^{-12}$.
We employ an admissible mesh with 4736 elements generated by the MATLAB 
command {\tt initmesh}, which produces Delauney meshes.
As initial data, piecewise linear functions that connect the boundary values
are chosen for the ion concentrations, while the initial potential is computed
from the Poisson equation using the initial concentrations as charge density. 

Figures \ref{fig.sol1} and \ref{fig.sol2} show the concentration profiles and
the electric potential after 50 and 1400 time steps, respectively. The equilibrium
is approximately reached after 1653 time steps. The profiles depicted in Figure
\ref{fig.sol2} are already very close to the stationary state and correspond 
qualitatively well to the one-dimensional stationary profiles presented in \cite{BSW12}.
We observe that during the evolution, sodium inside the channel is replaced by the 
stronger positively charged calcium ions. For higher
initial calcium concentrations, the calcium selectivity of the channel acts 
immediately.

\begin{figure}[htb]
\includegraphics[width=\textwidth]{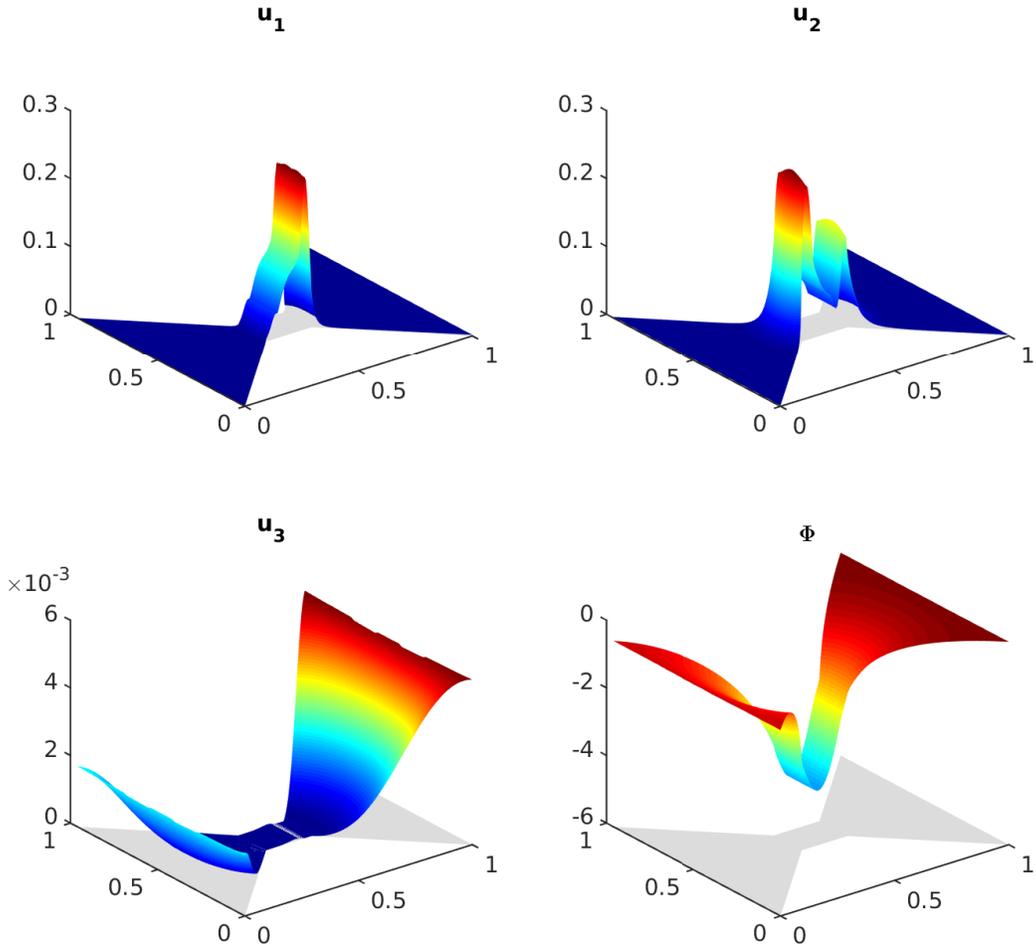}
\caption{Scaled concentrations of calcium, sodium, and chloride ions 
and electric potential after 50 time steps.}
\label{fig.sol1}
\end{figure} 

\begin{figure}[htb]
\includegraphics[width=\textwidth]{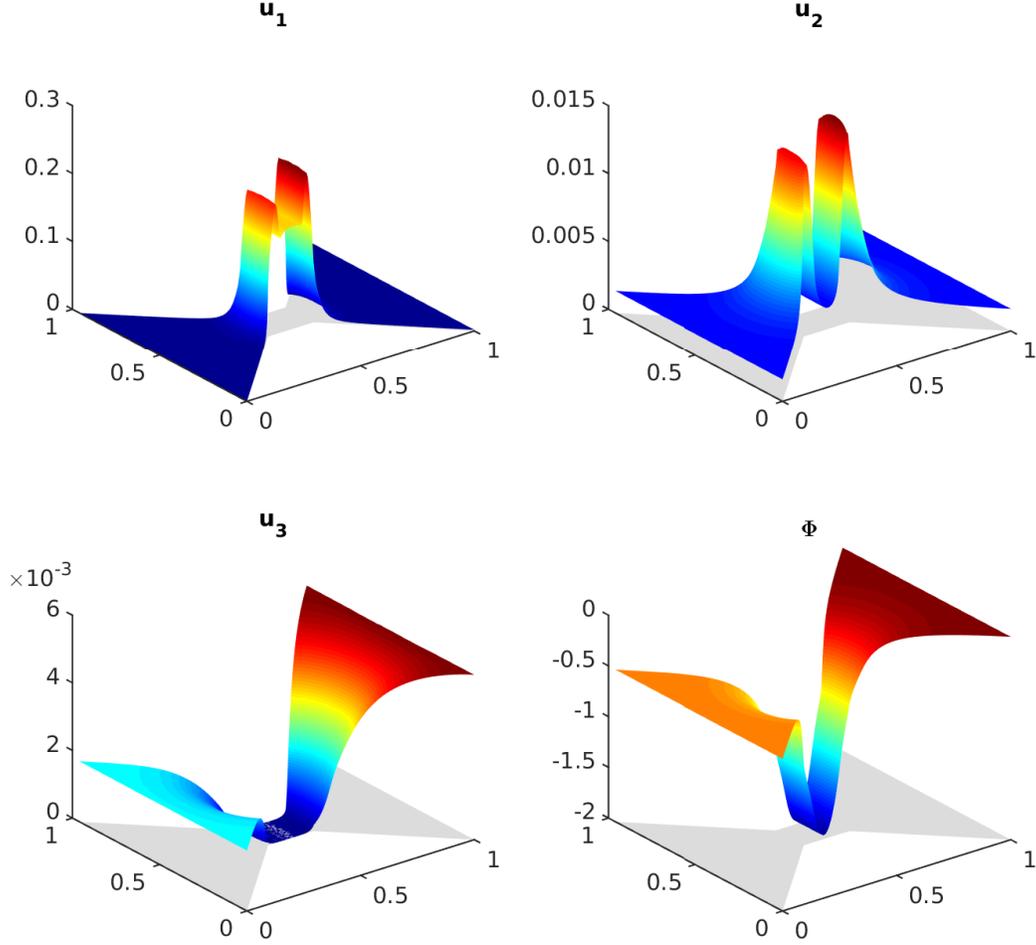}
\caption{Scaled concentrations of calcium, sodium, and chloride ions 
and electric potential after 1400 time steps (close to equilibrium).}
\label{fig.sol2}
\end{figure} 

The simulations suggest that  the solution tends towards a steady state as
$t\to\infty$. The large-time behavior can be quantified by computing the relative
entropy $E^k$ with respect to the stationary solution, where
$$
  E^k = \sum_{K\in\T}\m(K) \sum_{i=0}^n u_{i,K}^k 
	\log\bigg(\frac{u_{i,K}^k}{u_{i,K}^\infty}\bigg) 
	+ \frac{\lambda^2}{2}\sum_{\sigma\in\E}\tau_\sigma 
	\textrm{D}_{K,\sigma}(\Phi^k-\Phi^\infty)^2
$$
and $(u^\infty_{i,K},\Phi^\infty)$ is the constant steady state determined from
the boundary data. Figure \ref{fig.time} shows that the relative entropy 
as well as the discrete $L^1$ norms of the concentrations and electric potential
decay with exponential rate. Interestingly, after some initial phase, 
the convergence is rather slow and increases after this intermediate phase.
This phase can be explained by the degeneracy at $u_0=0$, which causes a small
entropy production slowing down diffusion. 
Indeed, as shown in \cite{GeJu17} for the one-dimensional setting, 
a small change in the oxygen concentration 
may prolong the intermediate phase of slow convergence drastically.

\begin{figure}[htb]
\includegraphics[width=\textwidth]{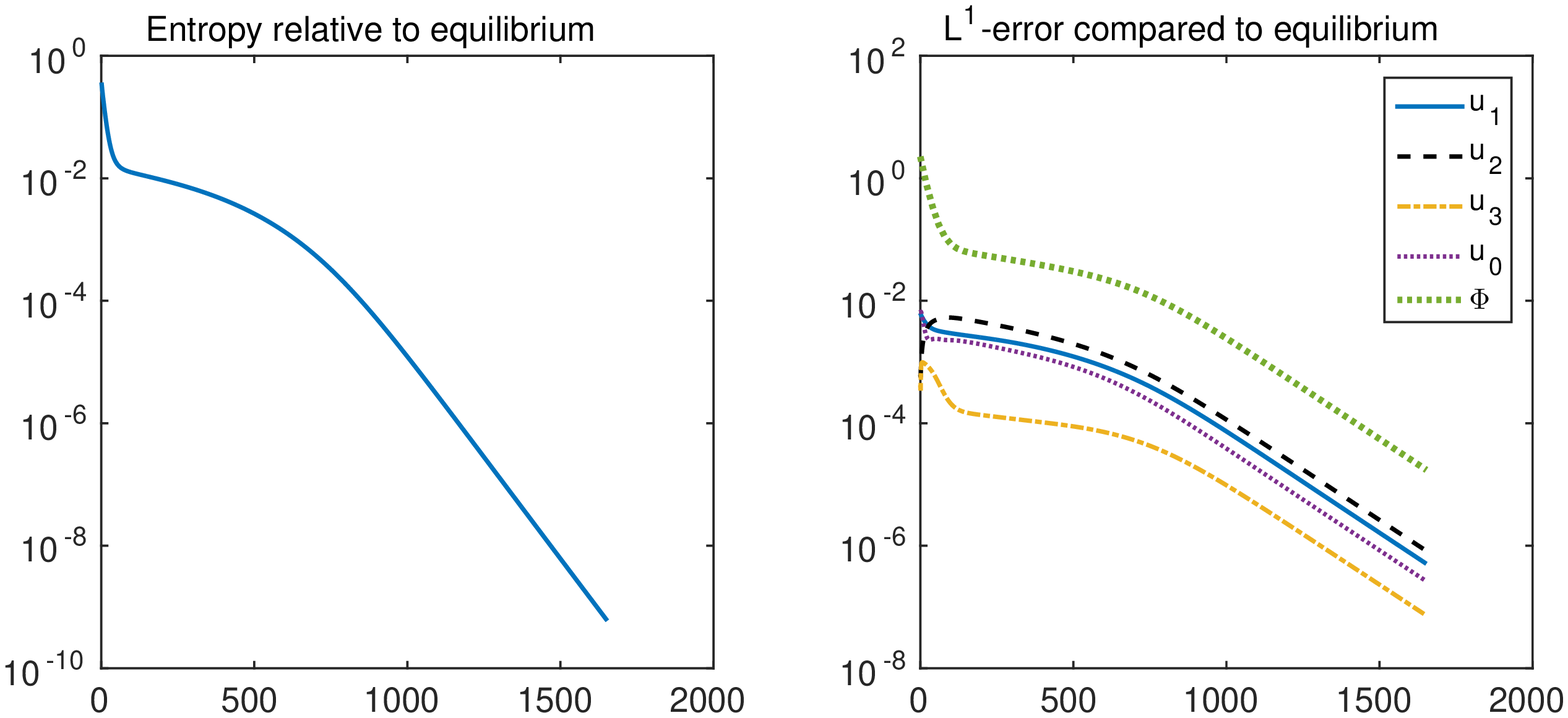}
\caption{Relative entropy (left) and discrete $L^1$ error relative to the equilibrium 
(right) over the number of time steps.}
\label{fig.time}
\end{figure} 

Since Assumptions (A1)-(A3) are not satisfied in our test case, the convergence
result of Theorem \ref{thm.conv} cannot be applied here. However, we still
observe convergence of the numerical solutions. As the exact solution is not
known explicitly, we compute a reference solution on a very fine mesh
with 75\,776 elements and mesh size $h(\T)\approx 0.01$. This mesh is obtained
from the coarse mesh by a regular refinement, dividing the triangles into
four triangles of the same shape. The reference solution is compared to 
approximate solutions on coarser nested meshes. In Figure \ref{fig.conv},
the errors in the discrete $L^1$ norm between the reference solution and the
solutions on the coarser meshes at two fixed time steps $k=50$ and $k=1400$
are plotted. We clearly observe the expected first-order convergence in space.

\begin{figure}[htb]
\includegraphics[width=\textwidth]{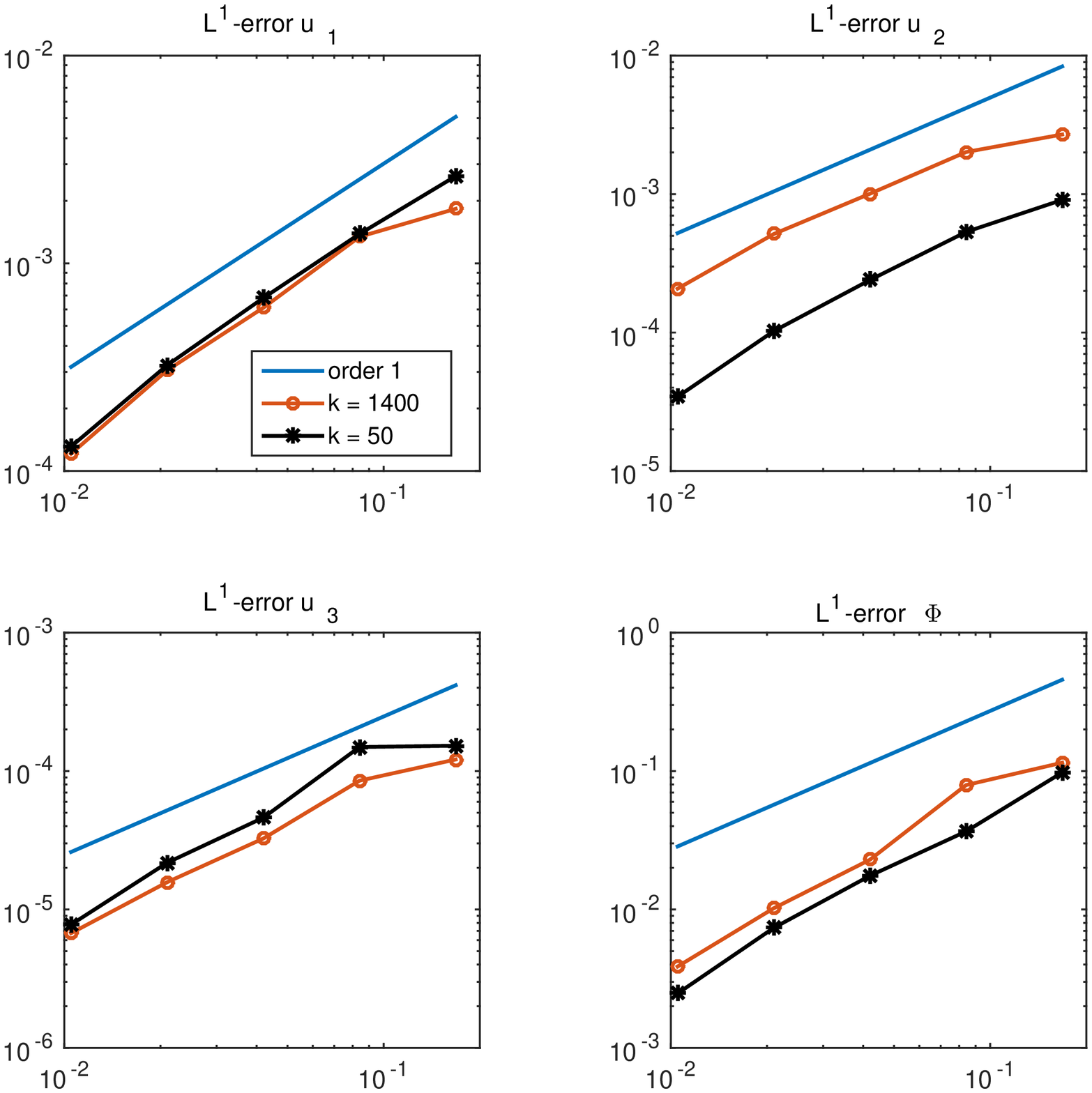}
\caption{Discrete $L^1$ error relative to the reference solution at two
different time steps over the mesh size $h(\T)$.}
\label{fig.conv}
\end{figure} 


\begin{appendix}
\section{Auxiliary results}\label{sec.app}

We prove two versions of discrete Aubin-Lions lemmas. The first one is 
a consequence of \cite[Theorem 3.4]{GaLa12}, the second one extends Lemma 13
in \cite{Jue15} to the discrete case. The latter result is new.
Recall that $\Omega_T=\Omega\times(0,T)$, $\na_m=\na_{\T_m,\dt_m}$ is
the discrete gradient defined in \eqref{5.grad}, and $\pa_t^\dt$ is the
discrete time derivative defined in \eqref{4.dtime}.

\begin{lemma}[Discrete Aubin-Lions]\label{lem.aubin}
Let $\|\cdot\|_{1,\T_m}$ be the norm on $\mathcal{H}_{\T_m}$
defined in \eqref{2.norm1} with the dual norm $\|\cdot\|_{-1,\T_m}$
given by \eqref{2.norm-1}, 
and let $(u_m)\subset\mathcal{H}_{\T_m,\dt_m}$ be a sequence
of piecewise constants in time functions with values in $\mathcal{H}_{\T_m}$
satisfying 
$$
  \sum_{k=1}^{N_m}\dt\big(\|u_m^k\|_{1,\T_m}^2 + \|\pa_t^{\dt_m} u_m^k\|_{-1,\T_m}^2\big)
	\le C,
$$
where $C>0$ is independent of the size of the mesh and the time step size.
Then there exists a subsequence, which is not relabeled, such that, as $m\to\infty$,
\begin{align*}
  u_m\to u &\quad\mbox{strongly in }L^2(\Omega_T), \\
	\na_m u_m \rightharpoonup \na u &\quad\mbox{weakly in }L^2(\Omega_T).
\end{align*}
\end{lemma}

\begin{proof}
The result is a consequence of Theorem 3.4 in \cite{GaLa12}. To apply this theorem,
we have to show that the discrete norms $\|\cdot\|_{1,\T_m}$
and $\|\cdot\|_{-1,\T_m}$ satisfy the assumptions of Lemma 3.1 in \cite{GaLa12}:
\begin{enumerate}
\item For any sequence $(v_m)\subset\mathcal{H}_{\T_m}$ such that there exists
$C>0$ with $\|v_m\|_{1,\T_m}\le C$ for all $m\in\N$, there exists 
$v\in L^2(\Omega)$ such that, up to a subsequence, $v_m\to v$ in $L^2(\Omega)$.
\item If $v_m\to v$ strongly in $L^2(\Omega)$ and 
$\|v_m\|_{-1,\T_m}\to 0$ as $m\to\infty$, then $v=0$.
\end{enumerate}
The first property is proved in, for instance, \cite[Lemma 5.6]{EGH08}.
Here, we need assumption \eqref{2.dd} on the mesh.
The second property can be replaced, according to \cite[Remark 6]{GaLa12},
by the condition that $\|\cdot\|_{1,\T_m}$ and $\|\cdot\|_{-1,\T_m}$ are dual norms
with respect to the $L^2(\Omega)$ norm, which is the case here. We infer that
there exists a subsequence of $(u_m)$, which is not relabeled, such that
$u_m\to u$ strongly in $L^2(\Omega_T)$. 
The weak convergence of the discrete gradients can be proved as in Lemma 4.4
in \cite{CLP03}. Indeed, the boundedness of $(\na_m u_m)$ in $L^2$ implies
the convergence to some function $\chi\in L^2(\Omega_T)$ (up to a subsequence). 
In order to show that $\chi=\na u$, it remains to verify that for all test functions
$\phi\in C_0^\infty(\Omega_T;\R^d)$,
$$
  \int_0^T\int_\Omega \na_m u_m\cdot\phi dxdt
	+ \int_0^T\int_\Omega u_m\diver\phi dxdt \to 0 \quad\mbox{as }m\to\infty.
$$
This limit follows from the definition of $\na_m u_m$ and the regularity of the
mesh. We refer to \cite[Lemma 4.4]{CLP03} for details.
\end{proof}

\begin{lemma}[Discrete Aubin-Lions of ``degenerate'' type]\label{lem.aubin2}
Let $(y_m)$ and $(z_m)$ be sequences in $\mathcal{H}_{\T_m,\dt_m}$ which are bounded in
$L^\infty(\Omega_T)$ and let $(y_m)$ be relatively compact in $L^2(\Omega_T)$, i.e.,
up to a subsequence, $y_m\to y$ strongly in $L^2(\Omega_T)$ and $z_m\rightharpoonup^* z$
weakly* in $L^\infty(\Omega_T)$. Furthermore, 
suppose that, for some constant $C>0$ independent of $m$,
$$
  \sum_{k=1}^{N_m}\dt_m\big(\|y_m^k\|_{1,\T_m}^2 + \|y_m^kz_m^k\|_{1,\T_m}^2
	+ \|\pa_t^{\dt_m} z_m^k\|_{-1,\T_m}^2\big)\le C.
$$
Then there exists a subsequence which is not relabeled such that
$y_mz_m\to yz$ strongly in $L^2(\Omega_T)$ as $m\to\infty$.

\end{lemma}

\begin{proof}
The idea of the proof is to use the Kolmogorov-Riesz theorem
\cite[Theorem 4.26]{Bre11} as in the continuous case; see
\cite[Section 4.4]{BDPS10} or \cite[Lemma 13]{Jue15}. 
The discrete case, however, makes necessary some changes in the calculations.
We need to show that
\begin{equation}\label{a.aux}
  \lim_{(\xi,\tau)\to 0}\int_0^{T-\tau}\int_\omega\big((y_mz_m)(x+\xi,t+\tau)
	- (y_mz_m)(x,t)\big)^2 dxdt = 0
\end{equation}
uniformly in $m$, where $\omega\subset\Omega$ satisfies $x+\xi\in\Omega$ for
all $x\in\omega$. First, we separate the space and time translation:
\begin{align*}
	\int_{0}^{T-\tau}\int_{\omega}
	&\Big((y_mz_m)(x+\xi,t+\tau)-(y_mz_m)(x,t)\Big)^2dxdt \\ 
	&\le 2\int_{0}^{T-\tau}\int_{\omega}
	\Big((y_mz_m)(x+\xi,t+\tau)-(y_mz_m)(x,t+\tau)\Big)^2dxdt \\
	&\phantom{xx}{}
	+ 2\int_{0}^{T-\tau}\int_{\omega}\Big((y_mz_m)(x,t+\tau)-(y_mz_m)(x,t)\Big)^2dxdt
	=: I_1+I_2.
\end{align*}
For the estimate of $I_1$, we apply a result for space translations of piecewise
constant functions $v$ with uniform bounds in the discrete $H^1(\Omega)$ norm, namely
$$
  \|v(\cdot+\xi)-v\|_{L^2(\Omega)}^2 \le |\xi|\big(|\xi| + Ch(\T)\big)
	\|v\|_{1,\T}^2
$$
for appropriate $\xi$, where $C>0$ only depends on $\Omega$
\cite[Lemma 4]{EGH99}. This shows that
$$
  I_1 \le C_1|\xi|\big(|\xi| + Ch(\T_m)\big)
$$
converges to zero as $\xi\to 0$ uniformly in $m$. 

For the second integral $I_2$, we write
\begin{align*}
	I_2 &\le 4\int_{0}^{T-\tau}\int_{\omega}z_m(x,t+\tau)^2\big(y_m(x,t+\tau)-y_m(x,t)
	\big)^2dxdt \\
	&\phantom{xx}{} + 4\int_{0}^{T-\tau}\int_{\omega}y_m(x,t)^2
	\big(z_m(x,t+\tau)-z_m(x,t)\big)^2dxdt =: I_{21}+I_{22}.
\end{align*}
The $L^\infty$ bounds on $z_m$ give
$$
  I_{21} \le C\int_0^{T-\tau}\int_\omega\big(y_m(x,t+\tau)-y_m(x,t)\big)^2 dxdt.
$$
By assumption, the sequence $(y_m)$ is relatively compact in $L^2(\Omega_T)$.
Therefore, we can apply the inverse of the Kolmogorov-Riesz theorem
\cite[Exercise 4.34]{Bre11} to conclude that $I_{21}$ converges to zero as $\tau\to 0$
uniformly in $m$. 

The analysis of $I_{22}$ is more involved. We split the integral in several parts:
\begin{align*}
	I_{22} &= \int_{0}^{T-\tau}\int_{\omega}y_m(x,t)^2z_m(x,t)
	\big(z_m(x,t)-z_m(x,t+\tau)\big)dxdt \\
	&\phantom{xx}{}+ \int_{0}^{T-\tau}\int_{\omega}y_m(x,t+\tau)^2z_m(x,t+\tau)
	\big(z_m(x,t+\tau)-z_m(x,t)\big)dxdt \\
	&\phantom{xx}{}+ \int_{0}^{T-\tau}\int_{\omega}\big(y_m(x,t)^2-y_m(x,t+\tau)^2\big)
	z_m(x,t+\tau)\big(z_m(x,t+\tau)-z_m(x,t)\big)dxdt \\
	&=: J_1+J_2+J_3.
\end{align*}
The first two integrals $J_1$ and $J_2$ are treated similarly as in 
\cite[Lemma 3.11]{BrMa13}. Indeed, let $\lceil s\rceil$ denote the smallest integer 
larger or equal to $s$. Defining $n_{m}(t):=\lceil t/\dt_m \rceil$, 
we can formulate
$$
	z_m(x,t+\tau)-z_m(x,t) = \sum_{k=n_{m}(t)+1}^{n_{m}(t+\tau)} 
	\big(z^k_{m,K}-z^{k-1}_{m,K}\big)
$$
for $x\in K$, $0\le t\le T-\tau$. With this formulation, we can bound $J_1$,
using the duality of $\|\cdot\|_{1,\T_m}$ and $\|\cdot\|_{-1,\T_m}$:
\begin{align*}
	J_1 &\leq \int_{0}^{T-\tau}\bigg(\sum_{K\in\T_m}\m(K)
	\big(y_{m,K}^{n_{m}(t)}\big)^2 z_{m,K}^{n_{m}(t)}
	\sum_{k=n_{m}(t)+1}^{n_{m}(t+\tau)} \big(z^{k-1}_{m,K}-z^k_{m,K}\big)\bigg)dt \\
	&\le \int_{0}^{T-\tau}\bigg(\sum_{k=n_{m}(t)+1}^{n_{m}(t+\tau)} 
	\big\|(y_{m}^{n_{m}(t)})^2 z_{m}^{n_{m}(t)}\big\|_{1,\T_m} 
	\big\|z^k_{m}-z^{k-1}_{m}\big\|_{-1,\T_m} \bigg)dt \\
	&\le \frac{1}{2}\int_{0}^{T-\tau}\sum_{k=n_{m}(t)+1}^{n_{m}(t+\tau)} 
	\dt_m\|(y_{m}^{n_{m}(t)})^2 z_{m}^{n_{m}(t)}\|_{1,\T_m}^2 dt \\
	&\phantom{xx}{}+ \frac{1}{2}\int_{0}^{T-\tau}
	\sum_{k=n_{m}(t)+1}^{n_{m}(t+\tau)} \frac{1}{\dt_m}
	\big\|z^k_{m}-z^{k-1}_{m}\big\|_{-1,\T_m}^2 dt \\
	&\le \frac{\tau}{2}\sum_{k=1}^{N_m}\dt_m\|(y_{m}^k)^2 z_{m}^k\|_{1,\T_m}^2 
	+ \frac{\tau}{2}\sum_{k=1}^{N_m}\frac{1}{\dt_m}
	\big\|z^k_{m}-z^{k-1}_{m}\big\|_{-1,\T_m}^2,
\end{align*}
where the last inequality follows from \cite[Lemmas 4.1 and 4.2]{ABH13}. 
Let us remark that, for all $\sigma =K|L\in\E_{\rm int}$, we can rewrite 
$$
(y_K)^2z_K-(y_L)^2z_L=\frac{y_K+y_L}{2}(y_Kz_K-y_Lz_L)+\frac{y_Kz_K+y_Lz_L}{2}(y_K-y_L).
$$
Then,
$$
\|(y_{m}^k)^2 z_{m}^k\|_{1,\T_m}^2 \leq \m(\Omega)\|(y_{m}^k)^2 z_{m}^k\|_{L^\infty(\Omega)}^2
+\|y_{m}^k\|_{L^\infty(\Omega)}\|y_{m}^kz_{m}^k\|_{1,\T_m}+ \|y_{m}^kz_m^k\|_{L^\infty(\Omega)}\|y_{m}^k\|_{1,\T_m}.
$$
Hence, $J_1\le C\tau$ for some $C>0$. An analogous estimation leads to $J_2\le C\tau$.
It remains to estimate the integral $J_3$. For this, we use, similar to the treatment
of $I_{21}$, the $L^\infty$ bounds on $y_m$ and $z_m$:
$$
  J_3 \le C\int_0^{T-\tau }\int_\omega\big|y_m(x,t+\tau)-y_m(x,t)\big|dxdt.
$$
This expression converges to zero uniformly in $m$ because of the relative compactness
of $(y_m)$ in $L^2(\Omega_T)$. 

We deduce from the previous computations that \eqref{a.aux} holds true. 
Therefore, the product $(y_mz_m)$ converges strongly in $L^2(\Omega_T)$, up to
some subsequence, and in view of the convergences $y_m\to y$ strongly in 
$L^2(\Omega_T)$ and $z_m\rightharpoonup^* z$ weakly* in $L^\infty(0,T;L^\infty(\Omega))$,
the limit of $(y_mz_m)$ equals $yz$, which finishes the proof.
\end{proof}

\end{appendix}


\end{document}